\documentclass[10pt,reqno]{amsart}

\usepackage{amsmath,amssymb,verbatim}
\usepackage[utf8]{inputenc}
\usepackage{enumerate}
\usepackage{pdflscape}  
\usepackage{amsthm}
\usepackage{mathrsfs}
\usepackage{color}
\usepackage[normalem]{ulem}
\usepackage{cancel}
\usepackage{tikz-cd}
\usepackage{tikz}
\usetikzlibrary{matrix}
\usepackage[all]{xy}
\usepackage{mathtools}
\usepackage{cleveref}
\usepackage{arydshln}

\usepackage{amsaddr}

\usepackage{comment}

\newtheorem{theorem}{Theorem}[section]
\newtheorem{lemma}[theorem]{Lemma}

\newtheorem{proposition}[theorem]{Proposition}
\newtheorem{corollary}[theorem]{Corollary}

\newtheorem{convention}{Convention}

\theoremstyle{definition}
\newtheorem{example}[theorem]{Example}
\theoremstyle{definition}
\newtheorem{examples}[theorem]{Examples}

\theoremstyle{definition}

\theoremstyle{remark}
\newtheorem{remark}[theorem]{Remark}

\newcommand{\D}{{\mathcal{D}}}

\newcommand{\Ker}{\mbox{\rm Ker }}
\newcommand{\Imagen}{\mbox{\rm Im }}

\newcommand{\Id}{\mbox{\rm Id}}
\newcommand{\E}{\mathcal{E}}

\newcommand{\End}{\mbox{\rm End}}

\newcommand{\Res}{\mbox{\rm Res}}
\newcommand{\Ind}{\mbox{\rm Ind}}
\newcommand{\FM}{\mbox{\rm FM}}
\newcommand{\Mod}{\mbox{\rm mod}}
\newcommand{\Modr}{\Mod\text{-}}

\newcommand{\gr}{\mbox{\rm gr}}
\newcommand{\grr}{\gr\text{-}}
\newcommand{\grl}{\text{-}\gr}

\newcommand{\Z}{{\mathbb Z}}

\newcommand{\res}{\mbox{res}}

\newcommand{\ind}{{\rm ind}}
\newcommand{\lRGX}{(R,G,X)\grl}
\newcommand{\rRGX}{\grr(R,G,X)}
\newcommand{\lSHY}{(S,H,Y)\grl}
\newcommand{\rSHY}{\grr(S,H,Y)}
\newcommand{\hotimes }{\widehat{\otimes}}

\newcommand{\matriz}[1]{\begin{array} #1 \end{array}}
\newcommand{\pmatriz}[1]{\left(\begin{array} #1 \end{array}\right)}

\DeclareMathOperator{\Hom}{Hom}

\newcommand{\qand}{\quad \text{and} \quad}
\newcommand{\Supp}{{\rm Supp}}

\numberwithin{equation}{section}

\newtheoremstyle{prop}
 {0.3em}
 {0.3em}
 {\itshape}
 {}
 {\bfseries}
 {.}
 {.5em}
 {}

\topskip=\baselineskip 
\title{Groupoid graded rings and  their categories of graded modules}

\author{Caio Antony}

\address{Departamento de Matem\'{a}tica, Instituto de Matemática e Estatística, Universidade de São Paulo, São Paulo, Brazil} \email{caioagma@gmail.com}

\author{Ángel del Río}

\thanks{The first author was financed, in part, by the São Paulo Research Foundation (FAPESP), Brazil. Process Numbers 2022/11166-6 and 2023/11994-9. 
The second author is partially supported by PID2024-155576NB-I00 funded by MICIU/ AEI/10.13039/501100011033/FEDER, UE and by Grant 22004/PI/22 of Fundación Séneca de la Región de Murcia.}

\address{Departamento de Matem\'{a}ticas, Universidad de Murcia, 30100, Murcia, Spain} \email{adelrio@um.es}

\begin{document}

\begin{abstract}
Let $G$ be a groupoid acting on a set $X$ and let $R$ be a $G$-graded ring with graded local units. 
We study the main properties of the category $\rRGX$ of $X$-graded $R$-modules and adjoint functors between categories of this kind. We characterize the latter in terms of tensor-like and hom-like functors. As an application we obtain a characterization of equivalences between such categories in the spirit of Morita theory. Then we introduce restriction and induction functors between categories of type $\rRGX$ and show that many functors between such categories can be realized naturally as a restriction or induction functor. This includes the forgetful functor, the functor associating a $G$ graded $R$-module with the $H$-graded module formed by the sum of the homogeneous components of degree in subgroupoid $H$, and the one associating it with the collapsing of homogeneous components in cosets modulo $H$ for $H$ a wide subgroupoid. We characterize when a restriction or induction functor is an equivalence of categories. Finally, we prove that $\rRGX$ is always equivalent to the category of modules over a ring with local units and, generalizing a result of Menini and N\u{a}st\u{a}sescu, we characterize when the category $\rRGX$ of  is equivalent to the category of modules of a unital ring.
\end{abstract}

\maketitle

\section{Introduction}

Gradings in rings and modules have a long tradition in mathematics starting from the most classical example of the natural $\Z$-gradings on polynomial rings, or more generally $\Z^n$-gradings on polynomial rings in $n$ variables. Those gradings are ubiquitous in algebraic geometry and served as a model for gradings by abstract groups and more generally by semigroups. They play an important role in several areas of mathematics as representation theory or ring theory, but also in physics or computational algebra. 

Most of the literature on graded rings is devoted to group graded unital rings  \cite{BahturinSehgalZaicev2001, Beattie1988,CohenMontgomery1984, Dade1980, DascalescuNastasescuNastasescu2023,FanKulshammer2000,GomezPardoNastasescu1991,GordonGreen1982, Marcus1999}. 
In that case some functors play a relevant role. One of them is the forgetful functor which associates a graded module with its ungraded version. Another one is the restriction functor that for a group $G$, a subgroup $H$ of $G$ and a $G$-graded ring $R$, associates each $G$-graded $R$-module $M=\oplus_{g\in G} M_g$ with $M_H=\oplus_{h\in H} M_h$, considered as $H$-graded $R_H$-module. 
These functors connect the categories $\grr R$ of $G$-graded right $R$-modules with the category  $\Modr R$ of (ungraded) right $R$-modules, and with the category $\grr R_H$ of $H$-graded $R_H$-modules.

In the last two decades of the 20th century, powerful categorical methods were developed in the theory of group graded rings. The book \cite{NastasescuVanOystaeyen} of C. N\u{a}st\u{a}sescu and F. Van Oystaeyen is a good introduction to these methods. 
Some attempts to extend these methods to the case where the grading group is replaced by a semigroup obtained some partial success \cite{AbramsMenini1996, AbramsMenini1999,AbramsMeninidelRio1994, ClaseThesis,ClaseJespersKelarevOkninski1995, Kelarev1998,KelarevOkninski}, but in general the results are not as satisfactory as in the group grading case. 

C. N\u{a}st\u{a}sescu, S. Raianu and F. Van Oystaeyen introduced in \cite{NastasescuRaianuVanOystaeyen} gradings by $G$-sets and, somehow surprisingly, in this case, the methods for group graded rings extended very nicely (see also \cite{BeattieDascalescu1996, BulacuDascalescuGrunenfelder1997,Rio1992}). Among other applications, this approach allowed an inductive method to  connect the graded and ungraded structure of a $G$-graded ring via a subnormal series of $G$. The idea consists in considering $G/H$ and $H\backslash G$ as $G$-sets, in the standard way, and then seeing a graded $R$-module as a $G/H$-graded $R$-module collapsing homogeneous component in the same coset. This can be considered as a partial forgetful functor with the full forgetting occurring for $H=G$. Alternatively, $G$ can be considered as an $H$ set and this gives another tool to pass from $R$ to $R_H$ and viceversa.

One of the good features of the category $\grr R$, for $R$ a group graded ring, is that it is a Grothendieck category with a projective generator. Moreover, it is equivalent to the category of modules of the smash product, which is a ring with local units, but not unital, in general (see \cite{CohenMontgomery1984} for the case where $G$ is finite, and the alternative approaches in \cite{Quin1985} and \cite{Beattie1988}, for the case of $G$ infinite). Furthermore, a $G$-grading can be seen as a coaction of a group algebra. Strictly speaking this only applies when $R_e$ is a field, or at least a commutative ring, but the Hopf algebra methods have been used in general without this assumption. This idea of using Hopf algebra methods appeared for the first time in \cite{CohenMontgomery1984} and have been largely used afterwards \cite{Beattie1988D,CaenepeelMilitaruZhuShenglin2002,DascalescuNastasescuNastasescu2014,MeniniZuccoli1997,Montgomery1993}.

In the latter years several authors have considered gradings by groupoids \cite{CalaLundstromPinedo2021,CalaLundstromPinedo2022,CristianoMarquesSanchez,LiuLi2006,Lundstrom2004,MoreiraOinert2024,NystedtOinertPinedo2018}. One of the motivations is its connection with partial actions \cite{Jerez2025}. The aim of this paper is to explore how far the categorical methods for group graded rings extend to the case of groupoid gradings. We present a general approach including gradings by $G$-sets, with $G$ a groupoid, and allowing the ring to have local units, in a grading sense. The latter is motivated, not only for the natural appearance of rings with local units already in the group grading case, as mentioned above, but also because some relevant examples of groupoid gradings appears naturally in rings with local units. This is the case, for example, of  groupoid rings with coefficients in a unital ring, or the ring of finite matrices over a unital ring indexed by an infinite set. We do not consider arbitrary rings, i.e. without local units, because even for the ungraded case the methods and results are of a completely different nature (see e.g. \cite{GarciaMarin1999,GarciaMarin2000,GarciaSimon1991}).
Observe that a groupoid graded ring can be seen as a contracted semigroup graded ring. 
However the existence of inverses allows to recover for groupoid graded rings methods of group graded rings which either fail or become cumbersome for semigroup graded rings 
(see e.g. \cite{AbramsMenini1996, AbramsMenini1999, AbramsMeninidelRio1994, ClaseThesis, ClaseJespersKelarevOkninski1995}). 

The paper is organized as follows: 
In \Cref{SectionNotation} we establish the basic notation, terminology and definitions.
It also includes the proof of a basic fact: If $G$ is a groupoid, $R$ is a $G$-graded ring with graded local units, and $X$ is a $G$-set, then the (right) $X$-graded $R$-modules form a Grothendieck category with a set of generators which are finitely generated projective and are defined in terms of shifts modules $(x)R$ and graded local units (\Cref{GCat}). 
We denote this category by $\rRGX$.
The connections between categories of this kind through some relevant functors is the subject of the paper. 
In \Cref{SectionAdjoints} we introduce hom-like and tensor-like functors on such categories and show that they realize all the adjoint functors between categories of $G$-set graded modules (\Cref{proptensadj}). 
Using this we obtain in \Cref{SectionEquivalences} a graded version of the Morita Theorem about equivalences between categories of modules over rings with identity (\Cref{Equivalencia}). 
In \Cref{SectionRestExt} we introduce a graded version of restriction and extension of scalars functors and prove that, as in the ungraded case, they form  adjoint pairs (\Cref{ExtensionRestrictionAdjoint}). We specialize this in several instances and show that forgetful and grading descend functors are examples of induction and restriction of scalars, respectively (see \Cref{exclassicalrestriction} and \Cref{ForgetfulExtension}). In the last section we obtain two applications. The first one characterizes when restriction and extension of scalars define an equivalence of categories (\Cref{EquivRest}). 
In the second application we first prove that $\rRGX$ is always equivalent to a category of (unital) modules over a ring with local units (\Cref{EquivRGXmodS}), and then we characterize when $\rRGX$ is equivalent to the category of modules of a unital ring (\Cref{MeniniNastasescu}). The former extends a theorem of M. Beattie \cite{Beattie1988} and the latter extends a theorem of C. Menini and C. N\u{a}st\u{a}sescu \cite{MeniniNastasescu1988}, which proved the same results for the category $\grr R(=\gr (R,G,G))$, for $R$ a unital group graded ring.

\section{\underline{Basic Notation and Definitions}}\label{SectionNotation}

\subsection{Actions of groupoids}

Recall that a \emph{groupoid} is a small category with invertible arrows.
We identify a groupoid with the set of its arrows together with their possible products. From this point of view a groupoid is a non-empty set $G$ together with four maps 
    $$\matriz{{rcc} G & \to & G \\ g & \mapsto & d(g)} \quad 
    \matriz{{rcc} G & \to & G \\ g & \mapsto & t(g)} \quad 
    \matriz{{rcc} G & \to & G \\ g & \mapsto & g^{-1}} \quad 
    \matriz{{ccl} \D_G=\{(g,h) \in G^2 : d(g)=t(h) \} & \to & G \\ (g,h) & \mapsto & gh}$$
satisfying the following conditions for every $g,h,k\in G$
\begin{itemize}
\item $d(d(g))=d(g)=t(g^{-1})$ and $t(t(g))=t(g)=d(g^{-1})$;
\item $(g,d(g)), (t(g),g)\in \D_G$ and $gd(g)=t(g)g=g$; 
\item $gg^{-1}=t(g)$ and $g^{-1}g=d(g)$;
\item if $(g,h), (h,k)\in \D_G$, then $d(gh)=d(h)$, $t(gh)=t(g)$ and $g(hk)=(gh)k$.
\end{itemize}
We use the notation $g:e\to f$ to abbreviate $e=d(g)$ and $f=t(g)$.
When $G$ is seen as a category, $d(g)$, $t(g)$ and $g^{-1}$ are the domain, codomain and inverse of $g$, for each $g\in G$, and if $(g,h)\in \D_G$, then $gh$ represent the product of $g$ and $h$. 
We are implicitly identifying each object of the original category with its identity. 
In this way the set of objects of the original category is identified with  
	\[G_0=\{e\in G : e=e^2\}=\{d(g):g\in G\}=\{t(g):g\in G\}.\]
Whenever we write $gh$ with $g,h\in G$, we are implicitly assuming that $(g,h)\in \D_G$ and in that case, we say that the product $gh$ \emph{is defined}. 
Other algebraic descriptions of groupoids can be found in \cite{AvilaMarinPinedo2020,Paterson1999}. 

\begin{remark}\label{GroupoidsSemigroups}
Groupoids can be seen  as a special type of semigroups. More precisely, if $G$ is a small category, not necessarily a groupoid, and we identify $G$ with its arrows, as previously explained for groupoids, then we define the semigroup $G^0=G\cup \{0\}$, where $0$ is a symbol not representing any arrow of $G$, and extend the product in $G$ by defining $ab=0$ whenever $g=0$, or $h=0$ or $g,h\in G$ and $gh$ is not defined in $G$. In case $G$ is a groupoid, then $G^0$ is an inverse semigroup with orthogonal idempotents (and in particular, with zero). Conversely, if $S$ is an inverse semigroup  with zero on which the idempotents are orthogonal, then the non-zero elements of $S$ with the non-zero products form a groupoid $G$ such that $G^0=S$. 

The construction is also valid with any semigroup $S$, i.e., one can construct a new semigroup $S^0$ as the disjoint union of $S$ and $\{0\}$, and the products on $S$ are extended to $S^0$ by setting $0s=s0=0$ for every $s\in S^0$.

\end{remark}

\begin{remark}
The product map of a groupoid determines uniquely the other maps. Indeed, $G_0=\{g\in G, g^2=g\}$ and if $g\in G$, then $d(g)$ is the unique $e\in G_0$ for which $ge$ is defined, $t(g)$ is the unique $e\in G_0$ such that $eg$ is defined, and  $g^{-1}$ is the unique element $h\in G$ satisfying $gh=t(g)$ and also the unique $k\in G$ such that $kg=d(g)$.
\end{remark}

A \emph{homomorphism of groupoids} is simply a functor between two groupoids. In our algebraic approach, a groupoid homomorphism is a map $\varphi$ between groupoids such that if $gh$ is defined, then so is $\varphi(g)\varphi(h)$ and $\varphi(gh)=\varphi(g)\varphi(h)$.

A \emph{left  action of a groupoid} $G$ is a covariant functor from $G$ to the category of sets. 
So a left action of $G$ on sets is formed by a collection of non-empty sets $(X_e)_{e\in G_0}$ and a collection of maps 
$(\chi_g:X_{d(g)}\to X_{t(g)})_{g \in G}$, such that 
    $$\chi_e=\Id_{X_e} \text{ for each } e\in G_0, \text{ and }\chi_{gh}=\chi_g \circ \chi_h, \text{ for each } (g,h)\in \D_G.$$ 
In order to recover the standard notation for actions of groups on sets we write
    $$g\cdot x=\chi_g(x), \quad (g\in G, x\in X_{d(g)}).$$
To stress the similarity with group actions we set
    $$X=\bigcup_{e\in G_0} X_e \qand \D_{G,X}=\{(g,x) : x\in X_{d(g)}\}$$
and we say that $X$ is a left $G$-\emph{set}. Observe that the $X_e$'s are not necessarily disjoint. So a left $G$-set is a set $X$ together with a map 
\[\D_{G,X} \mapsto X, (g,x)\mapsto g\cdot x\]
where $\D_{G,X}$ is a subset of $G\times X$ and setting $X_g=\{x\in X : (g,x)\in \D_{G,X}\}$ for $g\in G_0$, the following conditions are satisfied for every  $g\in G$ and $x\in X_g$ :
\begin{itemize}
	\item $X_g=X_{d(g)}$, 
	\item $g\cdot x\in X_{t(g)}$,	
	\item if $g\in G_0$, then $g\cdot x = x$, 
	\item $(hg)\cdot x = h\cdot (g\cdot x)$, for every $h\in G$ with $(h,g)\in \D_G$. 
\end{itemize}

The following notation for every $x\in X$ will be handy:
    $$E_x=\{e\in G_0 : x\in X_e\}.$$

A \emph{right action} of a groupoid $G$ is a contravariant functor from $G$ to the category of sets, i.e. for right actions
$\chi_g:X_{t(g)}\rightarrow X_{d(g)}$, $\chi_{gh}=\chi_h \chi_g$, and the product notation takes the form 
$$x \cdot g = \chi_g(x), \quad (g\in G, x\in X_{t(g)}).$$

Left and right actions of groupoids are equivalent concepts because one can pass from left to right action and viceversa by agreeing that 
$$g \cdot x = x\cdot g^{-1}, \text{ whenever } g\in G \text{ and } x\in X_{d(g))}.$$
So in the future we will talk of actions of groupoids on sets, without specifying the side of the action. 
We say that $g\cdot x$ (respectively, $x\cdot g$) \emph{is defined} if $x\in X_{d(g)}$ (respectively, $x\in X_{t(g)}$), equivalently if $d(g)\in E_x$ (respectively, $t(g)\in E_x$).
As for the product inside the groupoid, whenever we write $g\cdot x$ or $x\cdot g$, for $g\in G$ and $x\in X$, we are implicitly assuming that the product is defined.

The following examples describe groupoids and actions of groupoids on sets which occur naturally in several areas of mathematics.

\begin{examples}{\rm 
		\begin{enumerate}
			\item A group is simply a groupoid with a unique idempotent. In that case its actions are the classical actions of groups $G$ on sets.
			
			\item Let $T$ be a topological space. Then the fundamental groupoid $G$ of $T$ is formed by the homotopy classes of paths in $T$ \cite{Brown1988}. In this case $G_0$ can be identified with $T$ and $G$ acts on $T$ by setting $g\cdot x=y$ whenever $g$ is represented by a path from $x$ to $y$. 
			
			\item Let $V$ be a vector space and let $G$ be the set of injective linear maps $X\to V$ with $X$ a subspace of $V$. We endow $G$ with a structure of groupoid by claiming that $gh$ is defined when the image of $h$ is equal to the domain of $g$ and, in that case taking map composition as product in $G$.
			For this groupoid, $G_0$ is the set of inclusion maps into $V$ from subspaces of $V$; for $g:X\to V$ in $G$, $d(g)$ is the embedding map from $X$ to $V$, $t(g)$ is the embedding map from $g(X)$ to $V$, and $g^{-1}:g(X)\to V$ is given by setting $g^{-1}(g(y))=y$. Moreover, $G$ acts on $V$ by setting $V_g=X$ and $g\cdot x = g(x)$, for $g:X\to V$ in $G$. 
			
			This example can be generalized by replacing $V$ with a group, a module, a ring, an algebra, etc, and injective linear maps by injective homomorphisms from subgroups, submodules, subrings, subalgebras, etc.

			\item Let $(\Gamma,X,(X_g)_{g \in \Gamma}, (\theta_g)_{g \in \Gamma})$ be a partial action of a group as defined in \cite{Jerez2025}. 
			That is $\Gamma$ is a group, $X$ is a set and for every $g\in \Gamma$, $X_g$ is a subset of $X$  and  $\theta_g:X_{g^{-1}}\to X_g$ is a bijection such that $\theta_1$ is the identity map of $X$, and for every $g,h\in \Gamma$ and $x\in X_{h^{-1}}\cap X_{(gh)^{-1}}$,   $\theta_h(x)\in X_{g^{-1}}$ and $\theta_g(\theta_h(x))=\theta_{gh}(x)$. 
			We define a groupoid 
				\[G = \{(g,x): g \in \Gamma, x \in X_{g^{-1}} \}\]
			with 
				\[\D_G = \{((g,x),(h,y))\in G^2 : y\in X_{h^{-1}}\cap X_{g^{-1}h^{-1}} \text{ and } x=\theta_h(y)\}\]
			and product given as follows 
				$$(g,x)(h,y) = (gh,y).$$
			It is shown in \cite{Jerez2025} that every groupoid $G$  is of this form for some partial action of a group.
	\end{enumerate}}
\end{examples}

We now present some examples of actions for arbitrary groupoids. 

\begin{examples}\label{ExGroupoids}{\rm
Let $G$ be a groupoid. 
\begin{enumerate}
\item The \emph{trivial} $G$-set is formed by a set with a unique element, say $x$, with $g\cdot x=x$ for every $g\in G$.

\item\label{RegularAction} The product on $G$ defines two actions of $G$ on itself: The left and right \emph{regular} actions. 
For both $\D_{G,G}=\D_G$. Observe that in the left regular action $X_e=\{g\in G: t(g)=e\}$ and $E_g=\{d(g)\}$, while in the right one, $X_e=\{g\in G: d(g)=e\}$ and $E_g=\{t(g)\}$, for $e\in G_0$ and $g\in G$.

\item The \emph{action} of $G$ on itself \emph{by inner automorphisms} is defined by taking $X=\{x \in G : d(x)=t(x)\}$, as $G$-set with the action given by $g\cdot x = gxg^{-1}$, whenever $d(x)=t(x)=d(g)$. 

\item\label{subgroupoid} A \emph{subgroupoid} of $G$ is a subset of $G$ closed under products and inverses. 
If $H$ is a subgroupoid of $G$, then $H$ is a groupoid by itself in the natural way. 
    
    A subgroupoid $H$ is said to be \emph{wide} in $G$ if $G_0=H_0$. In that case the sets of the form $gH=\{gh : h\in H, d(g)=t(h)\}$ with $g\in G$, form a partition of $G$. Denoting $G/H=\{gH : g\in G\}$, and setting $g\cdot hH=(gh)H$ we get a (left) action of $G$ on $G/H$. 
    A right action of $G$ on $H\backslash G=\{Hg : g\in G\}$ is defined similarly.

    \item\label{fundamentalgroupoidaction} $X=G_0$ has a structure of $G$-set setting $X_e=\{e\}$ for each $e\in G_0$, and $g\cdot e=f$ whenever $g\in G$, $e=d(g)$ and $f=t(g)$.
\end{enumerate}}
\end{examples}

\subsection{Rings with local units}

In this paper, we do not assume that rings have a unit. Instead, we assume that every ring has local units, and their modules are unital. We remind the definitions of rings with local units and unital modules.

Let $R$ be a ring and let $\E(R)$ denote the set of idempotents of $R$. A right $R$-module is said to be \emph{unital} if  $MR=M$. A \emph{set of local units} of $R$ is a subset $U$ of $\E(R)$ satisfying the following condition: for every finite subset $F$ of $R$ there is $u\in U$ such that $r=ru=ur$ for every $r\in F$ (see e.g. \cite{AnhMarki1987}).
We say that $R$ \emph{has local units} if it has a set of local units, equivalently if $\E(R)$ is a set of local units of $R$.

We consider the following order in $\E(R)$:
    $$u\le v \quad \Leftrightarrow \quad uv=u=vu$$

The following obvious remark will be used frequently.

\begin{remark}\label{ArgumentoUnidades}
If $U$ is a set of local units of a ring $R$ and $M$ is a unital right $R$-module, then for every finite subset $F$ of $M$ there is $u\in U$ such that $mu=m$ for every $m\in F$. 
In particular, if $F$ is a finite subset of $\E(R)$, then there is $u\in U$ such that $e\le u$ for every $e\in F$.
\end{remark}

\subsection{Graded rings and modules}

Let $A$ be an abelian group and let $X$ be a non-empty set. An $X$-\emph{grading} of $A$ is simply a decomposition $A=\oplus_{x\in X} A_x$, with each $A_x$ a subgroup of $A$.
    In that case, every $a\in A$ has a unique expression as $a=\sum_{x\in X} a_x$, with $a_x\in A_x$ and the  \emph{support} of $a$ (with respect to the giving $X$-grading) is the finite set $\Supp(a)=\{x\in X : a_x\ne 0\}$. Moreover, $a_x$ is called the \emph{homogeneous component} of $a$ of degree of $x$, and $A_x$ is called the \emph{homogeneous component} of $A$ of degree $x$.  
    If $Z$ is a subset of $X$, then we set $a_Z=\sum_{z\in Z} a_z$ and $A_Z=\oplus_{z\in Z} A_z$.

\emph{In the remainder of the section $G$ is a groupoid and $R$ is a ring.}

A $G$-\emph{graded ring} is a ring $R$ with a $G$-grading $R = \oplus_{g\in G} R_g$ of its underlying additive group, such that 
	\begin{equation}\label{GradingCondition}
		R_g R_h\subseteq \begin{cases} 
			R_{gh}, & \text{if } gh \text{ is defined};\\ 
			\{0\}, & \text{otherwise}. \end{cases}
	\end{equation}
When $R_gR_h=R_{gh}$ whenever $gh$ is defined, we say that $R$ is \emph{strongly graded}.

We say that a $G$-graded ring $R$ has \emph{graded local units} if $\E(R_{G_0})$ contains a set of local units of $R$. Observe that this is equivalent to the following condition: for every $e\in G_0$, $R_e$ has local units and for every $g\in G$, $R_g$ is a unital module as left $R_{t(g)}$-module and as right $R_{d(g)}$-module. 

\begin{convention}
Throughout this paper all rings have local units, all graded rings have graded local units and all modules are unital.    
\end{convention}

From now on $R$ is a $G$-graded ring and $X$ is a $G$-set.
An $X$-\emph{graded right} $R$-\emph{module} is a right $R$-module $M$ with an $X$-grading $M=\oplus_{x\in X} M_x$ such that 
    $$M_x R_g \subseteq \begin{cases} 
    M_{x \cdot g}, & \text{if } x\cdot g \text{ is defined};\\ 
    \{0\}, & \text{otherwise}; \end{cases}$$
If $M$ and $N$ are $X$-graded right $R$-modules, then a homomorphism of $R$-modules $f:M\to N$ is said to be a \emph{graded homomorphism} if $f(M_x)\subseteq N_x$ for every $x\in X$. 
When we talk of $G$-graded right $R$-module we are considering $G$ as a $G$-set with the right regular action (\Cref{ExGroupoids}\eqref{RegularAction}).

Left $X$-graded $R$-modules and homomorphism between them are defined similarly, and for $G$-graded left $R$-modules, $G$ is considered as $G$-set with the left regular action.

Let now $H$ be a second groupoid, let $S$ be an $H$-graded ring, and let $Y$ be an $H$-set. 
An $(X,Y)$-\emph{bigraded} $(R,S)$-\emph{bimodule} is an $(R,S)$-bimodule $P$ together with an $X\times Y$-grading
    $$P = \bigoplus_{x\in X,y\in Y} {_xP_y}$$
such that ${_xP}=\oplus_{y\in Y} {_xP_y}$ is an $Y$-graded $S$-module for each $x\in X$, and $P_y=\oplus_{x\in X} {_xP_y}$ is an $X$-graded $R$-module for each $y\in Y$. 
If $R=S$, $G=H$ and $X=Y$ we simply say that $P$ is an $X$-\emph{bigraded} $R$-\emph{bimodule}.

\begin{convention}
We will denote the degree on the side where the ring acts, i.e. if $M$ is a graded left module, then the homogeneous component of degree $x$ is denoted $_xM$. Accordingly, if $M$ is a right (respectively, left) $X$-graded module and $m\in M$, then the homogeneous component of $m$ of degree $x$ is denoted $m_x$, (respectively, ${_xm}$). 
Similarly, if $P$ is an $(X,Y)$-bigraded bimodule, $p\in P$, $x\in X$ and $y\in Y$, then $_xp_y$ denotes the homogeneous component of $p$ of degree $(x,y)$.
\end{convention}

Let $M$ be a $G$-graded right $R$-module and let $x\in X$. 
The $x$-\emph{shift} of $M$ is the $X$-graded right $R$-module
$$(x)M = \bigoplus_{g\in G, t(g)\in E_x} M_g, \text{ with } 
    (x)M_y = \bigoplus_{g\in G, x\cdot g=y} M_g.$$
Similarly, if $M$ is a $G$-graded left $R$-module, then the $x$th shift of $M$ is the $X$-graded left $R$-module 
    $$M(x)=\bigoplus_{g\in G, d(g)\in E_x} M_g, \text{ with } 
    _yM(x) = \bigoplus_{g\in G, g\cdot x=y} M_g.$$    

\begin{remark}\label{CommonAMdR}
Recall that if $S$ is a semigroup then an $S$-graded ring is a ring $R$ together with an $S$-grading $R=\oplus_{s\in S} R_s$ such that $R_sR_t\subseteq R_{st}$ for every $s,t\in S$. If $S$ has a zero, i.e. an element $0$ such that $0s=s0=0$ for every $s\in S$, then an $S$-graded ring is said to be contracted if $S_0=0$. 
Observe that if $G$ is a groupoid then an $G$-graded ring is the same as a contracted $G^0$-graded ring (see \Cref{GroupoidsSemigroups}). 

Suppose that we are in the common ground of this paper with \cite{AbramsMeninidelRio1994}, i.e. suppose that $G$ is a groupoid, $X=G$ with left regular action (cf. \Cref{ExGroupoids}\eqref{RegularAction}) and $R$ is a unital $G$-graded ring considered as a contracted $G^0$-graded ring. 
Then $R(x)$ coincides with the module $P^x$ defined in \cite[Definition~2.9]{AbramsMeninidelRio1994}.
\end{remark}

\begin{examples}\label{ExGradedRings}{\rm
\begin{enumerate}
\item\label{GroupoidRing} If $G$ is groupoid and $R$ is a ring, then the groupoid ring $RG=\oplus_{g\in G} Rg$ is a $G$-graded ring in the natural way.
This ring is also the contracted semigroup ring of $G^0$ (see  \Cref{CommonAMdR}).
	
\item\label{Matrix} Let $I$ be a set and let $G=I\times I$. We endow $G$ with a structure of groupoid by declaring that the defined products are as follows:
        $$(i,j)(j,k)=(i,k), \text{ for } i,j,k\in I.$$
Consider the ring $\FM_I(R)$ formed by the $I\times I$ matrices $(a_{i,j})_{i,j\in I}$ with entries in $R$ and having finitely many non-zero entries. 
Declaring homogeneous of degree $(i,j)$, the matrices vanishing outside the $(i,j)$th entry defines a $G$-grading on $\FM_I(R)$.

Moreover, $I$ is a $G$-set with defined products as follows:
	\[(y,x)\cdot x = y\]
Then $M^{(I)}$, i.e. the direct sum of $|I|$ copies of $M$, is an $I$-graded right $\FM_I(R)$-module with the natural row-times-matrix product and $I$-grading. Observe that this is a unital module if and only if so is $M$. However, if $I$ is infinite, then the direct product $M^I$ of $|I|$ copies of $M$, has a natural structure of $\FM_I(R)$-module but it is neither unital nor graded.

\item Let $G$ be an arbitrary groupoid, let $R$ be a $G$-graded ring, let $H$ be a wide subgroupoid of $G$ and consider $G/H$ as a left $G$-set as in \Cref{ExGroupoids}\eqref{subgroupoid}. Then $R=\oplus_{A\in G/H} R_A$ defines a structure of $G/H$-graded left $R$-module on $R$.

\item\label{Rgorro} Let $R$ be a $G$-graded ring and let $X$ be a $G$-set. Then for every $x,y\in X$
    $$(x)R_y = \bigoplus_{g\in G, x\cdot g=y} R_g = \bigoplus_{g\in G, x=g\cdot y} R_g = {_xR(y)},$$
and we denote
    $$\widehat R = \bigoplus_{x\in X} (x)R = \bigoplus_{x\in X} R(x).$$
This is an $X$-bigraded $R$-bimodule with 
    $$_x\widehat R_y = {_xR(y)}= (x)R_y.$$
\end{enumerate}
}\end{examples}

Observe that the graded rings $\FM_I(R)$ and $RG$ in the first two items of \Cref{ExGradedRings} have graded local units if and only if $R$ has local units. In particular, if $R$ is unital, then both have graded local units, but they are not unital in general. These examples are our motivation for considering rings with local units instead of unital rings.

\subsection{Categories of graded modules}

In this section $G$ is a groupoid, $R$ is a $G$-graded ring, $U$ is a set of graded local units of $R$ and $X$ is a $G$-set. 

The category of $X$-graded right (unital) $R$-modules with graded homomorphism is denoted $\rRGX$. 
The category of $X$-graded left $R$-modules is defined analogously and denoted $\lRGX$. 
If $X=G$ with the regular right (respectively, left) action, then $\rRGX$ (respectively, $\lRGX$) is denoted $\grr  R$ (respectively, $R\grl$).
If $X$ is the trivial $G$-set, then $\rRGX=\Mod\text{-}R$, the category of right (unital) $R$-modules.

For every $u\in U$ and $x\in X$, let us denote
    $$u^x=\sum_{e\in E_x} u_e\in (x)R_x = {_xR(x)}.$$  
The following obvious observation will be useful: If $M\in\rRGX$ and $m\in M_x$, then $mu=mu^x$. In particular, if $mu=m$, then $mu^x=m$.

If $M$ is a right $R$-module, then $(Mu)_{u\in U}$ and $(Mv\to Mu,m\mapsto mu)_{u\le v}$ form an inverse system of additive groups; its inverse limit is
    $$\varprojlim_{u\in U} Mu = \left\{(m_u)_{u\in U} \in \prod_{u\in U} Mu: m_vu=m_u \text{ for every } u,v\in U \text{ with } u \le v\right \}.$$

\begin{lemma}\label{HomShift}
Let $M\in \rRGX$ and let $x\in X$. Then the map
    \[\Hom_{\rRGX}((x)R,M) \to \varprojlim_{u\in U} M_xu, \quad 
    f \mapsto (f(u^x))_{u\in U}\]
is an isomorphism of additive groups.
The inverse associates $(m_u)_{u\in U}$ with the morphism $f$ given by  $f(r)=m_ur$ if $r\in uR$.
\end{lemma}

\begin{proof}
The only non-obvious part of the lemma is that the map $f:\varprojlim_{u\in U} M_xu\to \Hom_{\rRGX}((x)R,M)$ defined in the last sentence is well defined. 
To prove that we observe that, by \Cref{ArgumentoUnidades}, for every $r\in R$ there is $u\in U$ with $ur=r$ and if $v$ is another element of $U$ with $vr=r$, then there is $w\in U$ with $u,v\le w$. Then     
$$m_w r = m_w (ur) = (m_wu)r = m_u r,$$
and analogously, $m_w r = m_v r$. So $m_ur=m_vr$.
\end{proof}

\begin{remark}
Let $M\in \rRGX$, $x\in X$ and $m\in M_x$. Then the element of $\Hom_{\rRGX}((x)R,M)$ associated to $(mu)_{u\in U}\in \varprojlim_{u\in U} M_xu$ in the isomorphism of \Cref{HomShift} is the map 
\begin{equation}\label{lambdam}
	\lambda_m:(x)R\to  M, \quad r  \mapsto  mr.
\end{equation}
In particular, if $g\in G$ and $x\in X_{d(g)}$, then $R_g\subseteq (g\cdot x)R_x$ and so, if $s\in R_g$ then
\begin{equation}\label{leftmultmap}
	\lambda_s\in \Hom_{\rRGX}((x)R,(g\cdot x)R).
\end{equation}

\end{remark}

\begin{theorem}\label{GCat}
Let $G$ be a groupoid, let $X$ be a $G$-set and let $R$ be a $G$-graded ring with a set of graded local units $U$. Then $\rRGX$ is a Grothendieck category with $\{(x)R : x\in X\}$ as a family of generators. Moreover, for every $x\in X$ and $u\in U$, $u\cdot (x)R=\{ur : r\in (x)R\}$ is a finitely generated projective object of $\rRGX$ and $(x)R=\varinjlim_{u\in U} u\cdot (x)R$.
\end{theorem}

\begin{proof}
Clearly $\Hom_{\rRGX}(M,N)$ is a subgroup of the additive group $\Hom_R(M,N)$, for every $M,N\in \rRGX$. 
Furthermore the module theoretical coproducts of objects in $\rRGX$ are  in $\rRGX$, and the kernel and cokernel of every morphism in $\rRGX$ are also in $\rRGX$. This fails for the product. Instead, the product in $\rRGX$ of $(M_i)_{i\in I}$ is $N=\left(\oplus_{x\in X} \prod_{i\in I} (M_i)_x\right)R$, with the following grading:
\begin{eqnarray*}
N_x &=& \left\{\sum_{j=1}^k n_j r_j : r_j\in R_{g_j}, n_j\in \prod_{i\in I} (M_i)_{x\cdot g_j^{-1}}  \right\} \\
&=& \left\{(m_i) \in \prod_{i\in I} (M_i)_x : \text{ there is } u\in U, m_i=m_iu \text{ for every } i\in I\right\}.
\end{eqnarray*}
This proves that $\rRGX$ is an abelian category with exact direct limits.

If $f: M \rightarrow N$ is a non-zero morphism, then there exists a homogeneous element $m \in M_x$ such that $f(m) \ne 0$. 
Then $\lambda_m\in \Hom_{\rRGX}((x)R,M)$ and $m=mu$ for some $u\in U$. 
Moreover, $(f\circ \lambda_m)(u)=f(mu)=f(m)\ne 0$, so that $f\circ \lambda_m\ne 0$.
This proves that $\{(x)R:x\in X\}$ is a set of generators of $\rRGX$.

Let $x\in X$ and $u\in U$. Then $uR_g\subseteq R_g$ for every $g\in G$ and hence $u(x)R$ is a submodule of $(x)R$. Furthermore, if $y\in X$ and $r\in (x)R$, then $(ur)_y=ur_y$. Thus $u(x)R$ is a graded submodule of $(x)R$.
Moreover, if $u\le v$, then $u(x)R\subseteq v(x)R$ and,
    $$(x)R = \bigcup_{u\in U} u(x)R = \varinjlim_{u\in U} u(x)R.$$

To prove that $u(x)R$ is finitely generated and projective, observe that $u^x\in (u(x)R)_x$ and $u(x)R=u^x(x)R$. Therefore, $u^x$ is a generator of $u(x)R$ and if $g:u(x)R\to N$ and $p:M\to N$ are morphisms in $\rRGX$ with $p$ surjective, then $g(u^x)=p(m)$ for some $m\in M_x$.
Let $f$ be the restriction of $\lambda_m$ to $u(x)R$. Then $f$ is a morphism in $\rRGX$ and $p\circ f=g$.
\end{proof}

\begin{remark}\label{RGorro}
By \Cref{GCat}, $\{u(x)R : x\in X, u\in U\}$ is a set of finitely generated projective generators of $\rRGX$ and each $(x)R$ is locally projective in the sense of \cite{AnhMarki1987}, i.e. is a direct limit of projective objects of $\rRGX$. 
Then the $X$-bigraded $R$-bimodule $\widehat{R}=\oplus_{x\in X} (x)R=\oplus_{x\in X} R(x)$ is a locally projective generator of both $\rRGX$ and $\lRGX$ (see \Cref{ExGradedRings}\eqref{Rgorro}).  
\end{remark}

\begin{remark}
Let $S$ be a semigroup and let $R$ be a unital $S$-graded ring. 
In \cite{AbramsMeninidelRio1994} it is proved that the category of $S$-graded $R$-modules is a Grothendieck category with a family $\{P^x : x\in S\}$ of finitely generated projective generators, where $P^x=1_{[x^{-1}x]}R$ and $1_{[x^{-1}x]}=\sum_{s\in S, xs=x} 1_s$. This result and \Cref{GCat} intersects in the case where $G$ is a groupoid, $S=G=X$ and $R$ has an identity. In that case the non-zero homogeneous components of $1$ form a complete set of orthogonal idempotents of $G$ and $\Supp(1)=\{e\in G_0 : R_e\ne 0\}$, so that $\{1_e : e\in G_0\}$ is a set of homogeneous local units. 
Moreover, $P^x=1_{d(x)}R=1_{d(x)} (x)R$, for every $x\in G$. 
\end{remark}

\section{\underline{Hom, tensor, and adjoint functors}\label{SectionAdjoints}}

In this section we define tensor-like functors and hom-like functors between categories of graded modules and prove that they cover all the adjoint functors between such categories. 

\subsection{Tensor-like functors}

\begin{convention}\label{convention}
In the remainder of the paper $G$ and $H$ are groupoids, $R$ is a $G$-graded ring, $S$ is an $H$-graded ring,  $X$ is a $G$-set and $Y$ is an $H$-set.
Moreover, we fix a set $U$ of graded local units of $R$ and a set $V$ of graded local units of $S$. 
\end{convention}

For $M \in \rRGX$ and $N \in (G, X, R)-gr$, let $M \hotimes _R N$ denote the additive subgroup of $M \otimes_R N$ generated by $\{m \otimes n : m \in M_x, n \in {_xN}\}$.
Observe that $M\hotimes_R-$ defines an additive functor from $\lRGX$ to the category of abelian groups and similarly $-\hotimes_R N$ defines an additive functor from $\rRGX$ to the category of abelian groups.

Let $P$ be an $(X,Y)$-bigraded $(R,S)$-bimodule.
For every $M\in \rRGX$, 
    $$M\hotimes_R P = \bigoplus_{y\in Y} M\hotimes_R P_y$$
defines a $Y$-graded right $S$-module.
This yields a functor
$$-\hotimes_R P:\rRGX\to \rSHY.$$
Similarly we have a functor
$$P\hotimes_S -:\lSHY\to \lRGX.$$

\begin{remark}\label{OtimesAssociative}
The bifunctor $\hotimes$ is associative in the following sense: If $M\in\rRGX$ and $N\in\lSHY$, then
		\[(M\hotimes_R P) \hotimes_S N\cong M\hotimes_R (P \hotimes_S N)\]
and this isomorphism is natural in $M$ and $N$.
\end{remark}

\begin{lemma}\label{TensorShift}
For every $x\in X$ and $N\in \lRGX$, the standard isomorphism $R\otimes N\to N$ restricts to an isomorphism $\phi_{x,N}:(x)R \hotimes _R N \simeq {_xN}$. In particular, $\widehat{R}\hotimes_R N\cong N$ in $\lRGX$.

Moreover, if $P$ is an $(X,Y)$-bigraded $(R,S)$-bimodule, then
$\phi_{x,P}$ is an isomorphism of $Y$-graded right $S$-modules.
\end{lemma}

\begin{proof}
First of all observe that the natural map 
		\[\phi:R\otimes_R M \to M, r\otimes m \mapsto rm\]
is an isomorphism of additive groups. Indeed, the surjectivity is a consequence of the assumption that $M$ is unital, and the injectivity follows from the assumption that $R$ has local units, because if $\sum_{i=1}^n r_i\otimes m_i$ belongs to the kernel of $\phi$, then taking an idempotent $u$ such that $ur_i=r_i$ for every $i$, it follows that
	\[\sum_{i=1}^n r_i\otimes m_i = \sum_{i=1}^n ur_i \otimes m_i=u\otimes \sum_{i=1}^n r_im_i=u\otimes 0=0.\]
	
So, it suffices to show that $\phi((x)R\hotimes _R N)={_xN}$. To that end, if $0\ne r\in R_g\cap (x)R_y$ and $n\in {_yN}$, then $x\cdot g=y$, so that $g\cdot y=x$ and hence $\phi(r\otimes n)=rn\in {_xN}$. Conversely, if $n\in {_xN}$, then $n=un$ for some $u\in U$. Moreover, $u^x\in (x)R_x$ and $un=u^xn$. Thus $u^x\otimes n\in (x)R\hotimes_R N$ and $n=\phi(u^x\otimes n)$.

The statement about $\phi_{x,P}$ follows by straightforward arguments.
\end{proof}

\begin{remark}\label{NaturalityTensorShift}
The inverse of $\phi_{x,N}$ maps $n\in {_xN}$ to $u^x\otimes n$, if $un=n$ with $u\in U$.

The isomorphism $\phi_{x,N}$ of \Cref{TensorShift} is natural in both variables in the following sense: If $f:N\to N'$ is a morphism in $\lRGX$, then it restricts to an additive morphism $f_x:{_xN}\to {_xN'}$ and the following diagram is commutative
$$\xymatrix{
(x)R\hotimes_R N \ar[d]_{1\otimes f} \ar[r]^{\quad\phi_{x,N}} & {_xN} \ar[d]^{f_x}\\
(x)R\hotimes_R N' \ar[r]_{\quad\phi_{x,N'}} & {_xN'}}$$
On the other hand, each $f\in \Hom_{\rRGX}((x)R,(x')R)$ induces an additive map $f':{_xN}\to {_{x'}N}$ defined by setting $f'(n)=f(u^x)n$ if $n\in {_xN}$, $u\in U$ and $un=n$. The proof that the map $f'$ is well defined is similar to that of \Cref{HomShift}. Then the following diagram is commutative:
\[\xymatrix{
	(x)R\hotimes_R N \ar[d]_{f\otimes 1} \ar[r]^{\quad \phi_{x,N}} & {_xN} \ar[d]^{f'}\\
	(y)R\hotimes_R N \ar[r]\ar[r]_{\quad \phi_{x',N}} & {_{x'}N}}\]
\end{remark}

The first part of \Cref{TensorShift} and \Cref{NaturalityTensorShift} yields the following result.

\begin{proposition}\label{RGorroEquiv}
$\widehat{R}\hotimes_R-$ (respectively, $-\hotimes_R \widehat{R}$)  is naturally isomorphic to the identity functor of $\lRGX$ (respectively, $\rRGX$).
\end{proposition}

Many of the concepts introduced for groupoid graded rings  make sense for semigroup graded rings. We present two examples showing that in this more general context our discussion fails from the beginning. For example, for a semigroup $S$ and an $S$-graded ring  $\oplus_{s\in S} (s)R$ and $\oplus_{s\in S} R(s)$ are different in general and \Cref{TensorShift} might fail.
We choose the examples having \Cref{GroupoidsSemigroups} in mind. Namely we consider contracted $S$-graded rings where $S$ is a semigroup with zero on which one of the following  two conditions holds but the other fails: (1) $S$ is an inverse semigroup; (2) the idempotents of $S$ are orthogonal.

\begin{example}\label{Ejemploeaf}[Non-inverse semigroup with orthogonal idempotents]
Consider the small category $C=\{\circlearrowright e\stackrel{a}{\longrightarrow} f\circlearrowleft\}$, i.e. $C$ has two objects $e$ and $f$ and $a:e\to f$ is the only non-identity arrow. Imitating the construction of the semigroup $G^0$ for $G$ a groupoid, we take the semigroup $S=C^0=\{e,a,f,0\}$, with $e$ and $f$ identified with its identities, products of compatible arrows in $C$ are products in $S$, and the remaining products are $0$.
Take a unital ring $K$. Then the contracted semigroup algebra of $S$ with coefficients in $K$ is $KC=Ke\oplus Ka\oplus Kf$ and can be identified in an obvious way with the ring $R$ of upper triangular $2\times 2$-matrices with entries in $K$ with the following grading
	\[R_e=\pmatriz{{cc} K & 0 \\ 0 & 0}, \quad R_a=\pmatriz{{cc} 0 & K \\ 0 & 0}, \quad 
	R_f=\pmatriz{{cc} 0 & 0 \\ 0 & K}.\]
Moreover, $N=R_a$ is a two-sided graded ideal which we consider as a $C$-graded left $R$-module with ${_eN}={_fN}=0$ and ${_aN=R_a}$. 
Observe that $1=e+f$ is the identity of $R$ and $\{e,f\}$ is a set of graded local units. 
Extending the definition of the shifts $(x)R$ and $\hotimes$ in the natural way, we have  
\[(e)R = R_e \oplus R_a, \quad (e)R_e=R_e, \quad (e)R_a=R_a, \quad (a)R=(a)R_a=R_f, \quad (f)R=(f)R_f=R_f.\] 
As $R$ is a unital ring, the natural homomorphism $R\otimes_R N\to N$ is an isomorphism. 
Thus 
	\[(e)R\hotimes_R N \cong R_e \; {_eN} + R_a \; {_aN}=0, \quad  
	(a)R\hotimes_R N \cong R_f \; {_aN} =0, 
	\quad (f)R \hotimes_R N \cong R_f \; {_fN}=0,\]
and hence, for the $C$-graded $R$-module $\widehat{R}_r=\oplus_{c\in C} (c)R$, we have $\widehat R_r \hotimes_R N=0$.
\\
We have emphasized that $\widehat{R}_r$ is a right $R$-module, because actually it is different from the $C$-graded left $R$-module $\widehat{R}_l=\oplus_{c\in C} R(c)$. Indeed, 
	\[R(e)={_eR(e)}=R_e, \quad R(a)={_aR(a)}=R_e, \quad 
R(f)=R_a\oplus R_f, \quad 
{_aR(f)}=R_a \quad 
{_fR(f)} = R_f,\]
so $\widehat{R}_r = R_e\oplus R_a \oplus R_f\oplus R_f \ne R_e\oplus R_e \oplus R_a \oplus R_f=\widehat{R}_l$.
\end{example}

\begin{example}[Inverse semigroup without orthogonal idempotents]
Consider the commutative semigroup $T=\{z,a,b\}$ with $a^2=a$, $b^2=b$ and all the other products are equal to $z$. 
Both $T$ and $S=T^0$ are inverse semigroups, but in $S$ the idempotents are not orthogonal, because $ab=z\ne 0$. 
Let $R$ be an $T$-graded ring (i.e. a contracted $S$-graded ring). Considering $T$ as a left and right $T$-set, we have $(x)R=R=R(x)$ for every $x\in T$ and the gradings are given by 
	\[(x)R_y = {_xR(y)} = \bigoplus_{t\in T, xt=y} R_t.\]
So the possible non-zero homogeneous components of the $(t)R$'s are as follows:
	\[(z)R_z = R, \quad 
	(a)R_z=R_z\oplus R_b, \quad (a)R_a=R_a, \quad
	(b)R_z=R_z\oplus R_a, \quad (b)R_b=R_b.\]
To fix ideas, suppose that $R=KT$, the semigroup ring of $T$ with coefficients in a unital ring $K$. Observe that $R$ is a unital ring with identity $1_R=a+b-z$. 
Moreover, $R_z$, $R_a$ and $R_b$ are unital rings and $R_z$ is both a unital $R_a$-module and a unital $R_b$-module. Hence the identities of $R_z$, $R_a$ and $R_b$ form a set of graded local units of $R$.
Therefore, the natural map $R\otimes_R M\to M$ is an isomorphism and as each $(s)R=R$, for any left $R$-module $M$, we have 
\[(z)R\hotimes_R M \cong {_zM}, \quad 
(a)R\hotimes_R M\cong (R_z\oplus R_a){_zM}+(R_a){_aM}, \quad 
(b)R\hotimes_R M\cong (R_z\oplus R_b){_zM}+(R_b){_bM}. \]
For example, if $M=R$, considered as left $R$-module, then 
\[\widehat{R}\hotimes_R R \cong 
R_z\oplus [(R_z\oplus R_a)R_z+R_aR_a] \oplus 
[(R_z\oplus R_b){_zM}+R_bR_b]=3R_z\oplus R_a\oplus R_b\ncong R.\]
\end{example}

\subsection{Hom-like functors}

We now define covariant hom like-functors from an $(X,Y)$-bigraded $(R,S)$-bimodule $P$: 
    $$H(P_S,-):\rSHY \to \rRGX \qand H({_RP},-):\lRGX \to \lSHY$$
by setting 
    $$H(P_S,N)=\left(\bigoplus_{x\in X} \Hom_{\rSHY}({_xP},N)\right)R \qand 
    H({_RP},M)=S\left(\bigoplus_{y\in Y} \Hom_{\lRGX}(P_y,M)\right),$$
for $M\in \rRGX$ and $N\in \lSHY$, with gradings 
\begin{eqnarray*}
	H(P_S,N)_x&=&\{f\in \Hom_{\rSHY}({_xP},N) : f=fu \text{ for some } u\in U\}, \\
	{_yH({_RP},M)}&=&\{f\in \Hom_{\lRGX}(P_y,M) : f=vf \text{ for some } v\in V\}.
\end{eqnarray*}
We consider $H(P_S,N)$ as a submodule of $\Hom_S(P,N)_R$ by identifying $\Hom_{\rSHY}({_xP},N)$ with $\{f\in\Hom_{\rRGX}(P,N) : f({_yP})=0 \text{ for every } y\ne x\}$. 
Similarly, $H({_RP},M)$ is identified with a submodule of ${_S\Hom_R(P,M)}$. 
We take these submodules, instead of the full set of homomorphism, to get unital modules.

\subsection{Adjoint functors}

\begin{theorem}\label{proptensadj}
    Let $F: \rRGX \rightarrow \rSHY$ be a covariant functor. Then, the following are equivalent.
    \begin{enumerate}
        \item $F$ is left adjoint.
        \item $F$ is right exact and preserves direct sums.
        \item $F$ is naturally isomorphic to $- \hotimes _R P$ for some $(X,Y)$-bigraded $(R,S)$-bimodule $P$. In that case $H(P_S,-)$ is a right adjoint of $F$.
    \end{enumerate}
\end{theorem}
\begin{proof}
(1) implies (2) is well known.

(3) implies (1) Let $P$ be an $(X,Y)$-bigraded $(R,S)$-bimodule, let $M\in \rRGX$ and let $N\in \rSHY$. For every $\mu \in \Hom_{\rSHY}(M\hotimes_R P, N)$,
$\nu\in \Hom_{\rRGX}(M,H(P_S,N))$, $p\in P$ and $m\in M$ set
    $$[\alpha(\mu)(m)](p) = \sum_{y \in Y}\sum_{x \in X} \mu(m_x \otimes {_xp_y}) \qand
    \beta(\nu)(m \otimes p) = \nu(m)(p).$$
This defines mutually inverse maps
    $$\alpha: \Hom_{\rSHY}(M\hotimes_R P,N) \rightarrow \Hom_{\rRGX}(M,H(P_S,N))$$
and
    $$\beta: \Hom_{\rRGX}(M,H(P_S,N)) \rightarrow \Hom_{\rSHY}(M \hotimes_R  P,N)$$
which are natural in $M$ and $N$.

(2) implies (3) Suppose that $F$ is right exact and preserves direct sums. 
Let $P=\oplus_{x\in X} F((x)R)$ with the $(X,Y)$-bigrading $_xP_y = F((x)R)_y$. 
This is clearly a $Y$-graded right $S$-module. Moreover, if $r\in R_g\subseteq (x)R$, then, by \eqref{leftmultmap}, $\lambda_r\in \Hom_{\rRGX}((x)R,(g\cdot x)R)$ and hence $F(\lambda_r)\in \Hom_{\rSHY}({_xP},_{g\cdot x} P)$. Then $rp=F(\lambda_r)(p)$ defines a structure of $X$-graded left $R$-module on $P$.
Now it is easy to prove that this defines a structure of $(X,Y)$-bigraded $(R,S)$-bimodule. Furthermore, by \Cref{TensorShift},
$$F((x)R) = {_xP} \simeq (x)R \hotimes_R P.$$
Since $\{(x)R:x\in X\}$ is a set of generators of $\rRGX$, for every $M\in\rRGX$ there is a short exact sequence in $\rRGX$
    $$\oplus_{i\in I} (x'_i)R \to \oplus_{j\in J} (x_j)R \to M \to 0$$
Moreover, by the naturality of the isomorphism ${_xP} \simeq (x)R \hotimes_R P$ on the first factor (see \Cref{NaturalityTensorShift}), and using that $F$ is right exact and preserves direct sums, we deduce that such short exact sequence yields a commutative diagram with vertical isomorphisms.
    $$\xymatrix{
\oplus_{i\in I} (x'_i)R \hotimes_R P \ar[d]_{\cong} \ar[r] &  \oplus_{j\in J} (x_j)R \hotimes_R P \ar[d]_{\cong} \ar[r] &  M \hotimes_R P \ar[r] &  0 \\
F\left(\oplus_{i\in I} (x'_i)R\right)
\ar[r] &  F\left(\oplus_{j\in J} (x_j)R\right)
\ar[r] &  F(M) \ar[r] &  0}$$
Then there is an isomorphism $M\hotimes_R P\cong F(M)$, which is natural in $M$, again by the naturality of the primitive isomorphisms.
\end{proof}

\begin{remark}\label{LeftAdjointP}
The proof of \Cref{proptensadj} is constructive because it shows that if $F:\rRGX\to\rSHY$ is a left adjoint functor, then $F$ is naturally isomorphic to $-\hotimes_R P$ with $P=\oplus_{x\in X} F((x)R)$. 
In the special case where $F$ is the identity functor of $\rRGX$, $P$ is the $X$-bigraded bimodule $\widehat R$ (see \Cref{RGorro}).
\end{remark}

The following corollary is an obvious consequence of \Cref{proptensadj}.

\begin{corollary}\label{Hrightadj}
The following are equivalent for a covariant functor $F:  \rSHY\rightarrow \rRGX$
    \begin{enumerate}
        \item $F$ is right adjoint.
        \item $F$ is naturally isomorphic to $H(P_S,-)$ for some $(X,Y)$-bigraded $(R,S)$-bimodule $P$.
    \end{enumerate}
\end{corollary}

\begin{example}\label{EjemploForgethful}
The forgetful functor $\mathcal{F}:\rRGX\to \Mod\text{-}R$ is obviously exact and preserves direct sums. Therefore $\mathcal{F}$ is naturally isomorphic to $-\hotimes_R P$ and its right adjoint is isomorphic to $H(P_R,-)$ where $P=\widehat{R}$,
considered as $(X,\{1\})$-bigraded module with $_xP_1=(x)R$. Here $\{1\}$ is the trivial $H$-set for any groupoid $H$.
Then for every $M\in \Mod\text{-}R$
    $$H(P_R,M)_x = H((x)R,M)=\{f\in \Hom_R((x)R,M) : f=fu \text{ for some } u\in U\} \cong MR_{E_x}.$$
The latter isomorphism associates $f\in H(P_R,M)_x$ with $f(u^x)$, if $fu=f$. Its inverse associates $m\in M$ with the map $\lambda_m$ defined in \eqref{lambdam}.
Therefore $\mathcal{F}$ is left adjoint to the functor which on objects acts as follows
	\[\mathcal{G}:\Mod\text{-}R\to \rRGX, \quad M \mapsto \mathcal{G}(M)=M^{(X)},\] 
with $\mathcal{G}(M)_x=\{(m_x)\in M^{(X)} : m_y=0 \text{ for every } y\in X\setminus \{x\}\}$ and the product is defined as follows $(m_x)_{x\in X}r=(m_{x\cdot g^{-1}}r)_{x\in X}$ for $r\in R_g$. The action of $\mathcal{G}$ on morphisms is componentwise.
\end{example}

\Cref{SectionRestExt} contains more examples of adjoint functors between categories of graded modules.

The following lemma plays an important role in \Cref{SectionEquivalences}.

\begin{lemma}\label{HExact}
The functor $H(P_S,-)$  is exact if and only if $u({_xP})$ is projective in $\rSHY$ for every $x\in X$ and $u\in U$.
\end{lemma}

\begin{proof}
Suppose first that $H(P_S,-)$ is exact. Let $\alpha:N\to N'$ be an epimorphism in $\rSHY$ and let $x\in X$, $u\in U$ and $\gamma\in \Hom_{\rSHY}(u({_xP}),N')$.
Then the map $\bar\gamma:P\to N'$, given by $\bar\gamma(p)=\gamma(u(_xp))$, is an element of $H(P_S,N')_x$.
Moreover, $H(P_S,\alpha):H(P_S,N)\to H(P_S,N')$ is surjective, by assumption.
Therefore there is $\gamma'\in H(P_S,N)_x$ such that $\alpha\gamma'=\bar\gamma$. Then the restriction $\gamma'_0$ of $\gamma'$ to $u({_xP})$ belongs to $\Hom_{\rSHY}(u({_xP}),N)$ and satisfies $\alpha \gamma'_0=\gamma$.
This proves that $u({_xP})$ is projective in $\lSHY$.

Conversely, suppose that $u({_xP})$ is projective in $\rSHY$ for every $x\in X$ and $u\in U$.
Let $\alpha:N\to N'$ as above and let $\gamma\in H(P,N')_x$. Then there is $u\in U$ such that $\gamma=\gamma u$. Let $\gamma_0$ be the restriction of $\gamma$ to $u({_xP})$. Clearly $\gamma_0\in \Hom_{\rSHY}(u({_xP}),N')$. By assumption, there is $\gamma'\in \Hom_{\rSHY}(u({_xP}),N)$ with $\alpha\gamma'=\gamma_0$. Then $\bar \gamma(p)=\gamma'(u(_xp))$ defines an element of $H(P_S,N)_x$ such that $\alpha \bar \gamma = \gamma$. This shows that $H(P_S,-)$ is right exact and hence it is exact by \Cref{Hrightadj}.
\end{proof}

\section{\underline{Equivalences of categories of graded modules}}\label{SectionEquivalences}

The aim of this section is to obtain a Morita-like theorem for the categories of graded modules in the setting of \Cref{convention}.
We start introducing a graded version of the natural ring homomorphisms $R\to \End(P_S)$ and $S\to \End(_RP)$ associated to an $(R,S)$-bimodule $P$.

Let $P$ be an $(X,Y)$-bigraded $(R,S)$-bimodule.
Then $H(P_S,P)$ is a subring and an $R$-bisubmodule of $\End_S(P)$. Moreover, it is unital and $X$-graded as right $R$-module but, in general, it is neither unital nor $X$-graded as left $R$-module.
A natural bisubmodule which is $X$-bigraded and unital on both sides can be obtained by taking
\begin{eqnarray*}
E(P_S)&=&\{f\in R\cdot H(P_S,P) : f(P)\subseteq {_FP} \text{ for some finite subset } F \text{ of } X\} \\
&=& R\left(\bigoplus_{x,x'\in X}\Hom_{\rSHY}({_{x'}P},{_xP})\right)R.
\end{eqnarray*}
For the second equality we are identifying $\Hom_{\rSHY}({_{x'}P},{_xP})$ with the set of $S$-endomorphisms of $P$ with kernel containing $\oplus_{z\in X\setminus \{x'\}} ({_zP})$ and image contained in $_xP$.
The $X$-bigrading is given by
\begin{eqnarray*}
_{x'}E(P_S)_x &=&
\{f\in R\cdot H(P_S,P)_{x'} : f(P)\subseteq P_{x}\} \\
&=&\{f\in \Hom_{\rSHY}({_{x'}P},{_xP}) : \text{there is } u\in U \text{ such that } f=ufu\}.
\end{eqnarray*}

Analogously we define
\begin{eqnarray*}
E({_RP})&=&\{f\in H({_RP},P)\cdot S : f(P)\subseteq P_F \text{ for some finite subset } F \text{ of } Y\} \\
&=& S\left(\bigoplus_{y,y'\in X}\Hom_{\rRGX}(P_y,P_{y'})\right)S,
\end{eqnarray*}
with $Y$-bigrading
\begin{eqnarray*}
_yE(P_S)_{y'} &=&
\{f\in {_yH({_RP},P)}\cdot S : f(P)\subseteq P_{y'}\} \\
&=&\{f\in \Hom_{\rRGX}(P_y,P_{y'}) : \text{there is } v\in V \text{ such that } f=vfv\}.
\end{eqnarray*}

Consider the maps
$$\lambda_P:\widehat{R}\to E(P_S) \qand \varrho_P:\widehat{S}\to E({_RP})$$
given by
$$\lambda_P(r)(p)=r(_{x'}p) \qand \varrho_P(s)(p)=(p_y)s.$$
for $r\in {_x\widehat{R}_{x'}}$, $s\in {_y\widehat{S}_{y'}}$ and $p\in P$.

The following lemma together with the last comment in \Cref{RGorroEquiv} shows that, in the graded setting, $\widehat{R}$ plays the role that the regular bimodule $R$ plays in the ungraded case.

\begin{lemma}\label{lambdalemma}
$\lambda_{\widehat{R}}$ and $\varrho_{\widehat{R}}$ are isomorphisms of $X$-bigraded $R$-bimodules.
\end{lemma}

\begin{proof}
We prove the statement for $\lambda_{\widehat{R}}$;
the proof for $\varrho_{\widehat{R}}$ is analogous. 
Proving that $\lambda_{\widehat{R}}$ is a homomorphism of $R$-bimodules is straightforward.
To prove injectivity, take $0 \ne r=\sum_{x\in F} r_x$ with $F$ a finite set of $X$ and $r_x\in \widehat{R}_x$ for every $x\in F$.
By \Cref{ArgumentoUnidades}, there exists $u \in U$ such that $r_xu = r_x$ for every $x\in F$. Moreover, $r_xu=r_xu^x$ and $u^x \in (x)R={_x\widehat{R}}$. Let $\widehat{u}=\sum_{x\in F} u^x \in \oplus_{x\in F}\;{_x\widehat{R}}$. Then $\lambda_{\widehat{R}}(r)(\widehat{u}) = \sum_{x\in F} r_xu^x = r\ne 0$. Thus $\lambda_{\widehat{R}}$ is injective.
For the surjectivity, take $f \in {_xE(\widehat{R}_R)}_{x'}$.
Then there is $u \in U$ such that $fu = f$.
For every $a \in (x')R$,
$$f(a) = f(ua) = f(u^{x'}a) =f(u^{x'})a = \lambda_{\widehat{R}}(f(u^{x'}))(a),$$
i.e. $f=\lambda_{\widehat{R}}(f(u^{x'}))$.
This proves that $\lambda_{\widehat{R}}$ is surjective.
\end{proof}

\begin{lemma}\label{fingen}
Let $P$ be an $(X,Y)$-bigraded $(R,S)$-bimodule such that $u({_xP})$ is finitely generated as a right $S$-module for all $x \in X$ and $u \in U$. Then 
\begin{enumerate}
    \item\label{fingenCommutes} $H(P_{S}, -)$ commutes with direct sums and
    \item\label{fingenE} $E(P_S) = H(P_S,P)$.
\end{enumerate}
\end{lemma}
\begin{proof}
(1)     Let $(N_{i})_{i \in I}$ be a family of $Y$-graded right $S$-modules.
Let $f \in H(P_S, \oplus_{i \in I} N_i)$ and for every $x\in X$, let $f_x$ denote the homogeneous component of $f$ of degree $x$.
By the definition of $H(P_S,-)$, there exists a finite subset $F$ of $X$ and $u\in U$ such that $f_x=0$ for all $x\in X\setminus F$ and $f=fu$.
The hypothesis implies that $u(_FP)$ is finitely generated and hence $f(u(_FP))\subseteq \oplus_{j\in J} N_j$ for a finite subset $J$ of $I$. 
Thus the restriction of $f$ to $u(_FP)$ is of the form $\sum_{j\in J} f'_j$, with $f'_j\in \Hom_{\rSHY}(u(_FP),N_j)$ for every $j\in J$.
Define $\overline{f_j}:P\to N_j$, by setting $\overline{f_j}(p)=f'_j(u\sum_{x\in F} {_xp})$ for $p\in P$.
Since $f=fu$ and $f$ vanishes in $_{X\setminus F}P$, it follows that  $f=\sum_{j\in J} \overline{f_j} \in \oplus_{j\in J} H(P_S,N_j) \subseteq \oplus_{i\in I} H(P_S,N_i)$.
This proves that $H(P_S, \oplus_{i \in I} N_i)\subseteq \oplus_{i\in I} H(P_S,N_i)$. As the other inclusion is clear, we conclude that $H(P_S,-)$ commutes with direct sums.

(2) It suffices to show that $H(P_S,P)_{x}\subseteq {E(P_S)}_{x}$, for any $x \in X$.
Let $f \in H(P_S,P)_{x}$.
By (1), $f(P)\subseteq {_FP}$ for some finite subset $F$ of $X$. 
So it suffices to show that $f\in R\cdot H(P_S,P)$. 
Indeed, let $u\in U$ with $f=fu$. 
By assumption, $u({_xP})$ has a finite generating set, say $A$.
Then there is $u' \in U$ such that $u'f(a) = f(a)$ for each $a\in A$, and thus, $f=u'f \in R\cdot H(P_S,P)$, as desired.
\end{proof}

Combining \Cref{GCat} and \Cref{fingen} we get
\[E(\widehat{R}_R)= H(\widehat{R}_R,\widehat{R}) \qand E(_R\widehat{R})= H(_R\widehat{R},\widehat{R}).\]
Actually these two rings are isomorphic to the following subring of the ring $\FM_X(R)$: 
	\[S_X(R) = \{(a_{x,x'})_{x,x'\in X}\in \FM_X(R) : a_{x,x'} \in (x)R_{x'}, \text{ for every } x,x'\in X\},\]
where $\FM_X(R)$ denotes the ring of $X\times X$-indexed matrices over $R$ having finitely many non-zero entries.

\begin{lemma}\label{HRRMatrix}
	The ring $S_X(R)$ is isomorphic to both $E(\widehat{R}_R)$ and $E(_R\widehat{R})$.
\end{lemma}
\begin{proof}
We prove that $S_X(R)$ is isomorphic to $E(\widehat{R}_R)$; the proof that it is isomorphic to $E(_R\widehat{R})$ is similar. 
	
Let $\alpha \in E(\widehat{R}_R)$. Then $\alpha = \sum_{x \in X} \alpha_x$, with $\alpha_x\in \Hom_{\rRGX}((x)R,\widehat{R})$ and $\alpha_x=0$ for almost all $x\in X$  and there is $u\in U$ such that $\alpha_xu=\alpha_x$ for each $x \in X$. 
Recall that $u^x\in (x)R_x$ and $\alpha_xu^x=\alpha_x$ for every $x\in X$. Then $ \alpha_x(u^x) = \sum_{x' \in X} {_{x'}(\alpha_x(u^x))}$, with ${_{x'}(\alpha_x(u^x))} \in (x')R_x$ for every $x,x'\in X$. Define the map
	\[\Phi:  E(\widehat{R}_R) \to  S_X(R), \quad \alpha  \mapsto  ({_{x'}\alpha_x(u^x)})_{x',x \in X}.\]
It is easy to see that this map does not depend on the choice of $u$. 
To prove that $\Phi$ is a ring homomorphism, let $\alpha, \beta \in E(\widehat{R}_R)$ and $u,v\in U$ such that $\alpha_x u = \alpha_x$ and $\beta_x v = \beta_x$ for every $x\in X$. Then,
 $$\Phi(\alpha \beta) = ({_{x'}(\alpha \beta)_x(v^x)})_{x',x} = \left(\sum_{z \in X} {_{x'}\alpha_z(u^z)}{_z\beta_x(v^x)}\right)_{x',x} = \left(({_{x'}\alpha_x(u^x)})_{x',x}\right)\left(({_{x'}\beta_x(v^x)})_{x',x}\right) = \Phi(\alpha)\Phi(\beta).$$
	The inverse of $\Phi$ is given by $\Psi$
	\[\Psi:  S_X(R) \to  E(\widehat{R}_R), \quad  (a_{x',x})_{x',x \in X}  \mapsto  \alpha.\]
	such that, $\alpha(r) = \sum_{x' \in X} a_{x',x}r$, for $r\in (x)R$.
\end{proof}

The isomorphism $S_X(R)\cong E(\widehat{R}_R)$ endows $\widehat{R}$ with a natural structure of $(S_X(R),R)$-bimodule on which the product by elements of $S_X(R)$ correspond the standard matrix-by-column matrix product when each $\sum_{x\in X} r_x\in \widehat{R}$, with $r_x\in (x)R$, is identified with the column matrix $(r_x)_{x\in X}$. Considering $S_X(R)$ as a $T$-graded ring, where $T$ denotes the trivial group, we can see $\widehat{R}$ as a $(T,X)$-bigraded  $(S_X(R),R)$-bimodule, which we denote $Q$ to distinguish it from its structure of $(X,X)$-bigraded $(R,R)$-bimodule. Then $\lambda_Q:S_X(R)\to E(Q_R)$ is precisely the isomorphism in the proof of \Cref{HRRMatrix}.

\begin{lemma}\label{SLocalUnits}
The ring $S_X(R)$ has a set of local units $V$ such that  $v\widehat{R}$ is finitely generated as right $R$-module for every $v\in V$.
\end{lemma}

\begin{proof}
For every $u\in U$ and every finite subset $F$ of $X$ define $v_{u,F}=(v_{x,x'})_{x,x'\in X}$ with 
\[v_{x,x'}=\begin{cases} u^x, & \text{if } x=x'\in F; \\ 0, & \text{otherwise}.\end{cases}\]
Then $v_{u,F}$ is an idempotent of $S_X(R)$ and $v_{u,F}\widehat{R} \cong \oplus_{x\in F} u(x)R$, which is finitely generated by \Cref{GCat}. 

Now let $a=(a_{x,x'})_{x,x'\in X}\in S_X(R)$. Then there is a finite subset $F$ of $X$ such that if $a_{x,x'}\ne 0$, then $x,x'\in F$ and some $u\in U$ such that $a_{x,x'}=ua_{x,x'}u$ for every $x,x'\in X$. Let $v=v_{u,F}$. 
The $(x,x')$th entry of $va$ is 
\[\sum_{z\in F} u_{x,z} a_{z,x'} = \begin{cases} u^x a_{x,x'} = ua_{x,x'}=a_{x,x'}, & \text{if } x\in F; \\ 0=a_{x,x'}, & \text{otherwise}. \end{cases}\]
This shows that $va=a$. Similarly $av=a$. Thus the set formed by the $v_{u,F}$'s with $u\in U$ and $F$ a finite subset of $X$ is a set of local units of $S_X(R)$.
\end{proof}

The following theorem characterizes the equivalences of categories of graded modules in the context of \Cref{convention}.

\begin{theorem}\label{Equivalencia}
Let $G,H,X,Y,R$ and $S$ as in \Cref{convention} and let $F:\rRGX\to\rSHY$ be a functor.
Then the following conditions are equivalent.
\begin{enumerate}
    \item\label{EquiF} $F$ is an equivalence of categories.
    \item\label{EquiPQ} There are an $(X,Y)$-bigraded $(R,S)$-bimodule $P$ and a $(Y,X)$-bigraded $(S,R)$-bimodule $Q$ such that $\widehat{R} \simeq P \hotimes_{S}Q$ and $\widehat{S} \simeq Q \hotimes_{R} P$, as bigraded bimodules, and $F$ is naturally isomorphic to $-\hotimes_R P$.
    \item\label{EquiQ} There is a $(Y,X)$-bigraded $(S,R)$-bimodule $Q$ such that $F$ is naturally isomorphic to $H(Q_R,-)$, $\lambda_Q:\widehat{S}\to E(Q_R)$ is an isomorphism of $Y$-bigraded $S$-bimodules, $Q_R$ is a generator of $\rRGX$ and $v(_yQ)$ is finitely generated and projective for every $y\in Y$ and $v\in V$.
    \item\label{EquiP} There is an $(X,Y)$-bigraded $(R,S)$-bimodule $P$ such that $F$ is naturally isomorphic to $-\hotimes_RP$, $\lambda_P:\widehat{R}\to E(P_S)$ is an isomorphism of $X$-bigraded $R$-bimodules, $P_S$ is a generator of $\rSHY$ and $u({_xP})$ is finitely generated and projective for every $x\in X$ and $u\in U$.
\end{enumerate}
In that case
\begin{enumerate}[(a)]
    \item\label{EquiQI} $H(_SQ,-):\lSHY\to\lRGX$ is an equivalence of categories $\varrho_Q:\widehat{R}\to E(_SQ)$ is an isomorphism of $X$-bigraded $R$-bimodules, $_SQ$ is a generator of $\lSHY$ and $Q_xu$ is finitely generated and projective for every $x\in X$ and $u\in U$.
    \item\label{EquiPI} $P\hotimes_S-:\lSHY\to\lRGX$ is an equivalence of categories $\varrho_P:\widehat{S}\to E({_RP})$ is an isomorphism of $Y$-bigraded $S$-bimodules, $_RP$ is a generator of $\lRGX$ and $(P_y)u$ is finitely generated and projective for every $y\in Y$ and $v\in V$.
\item\label{EquiIzq} $H(_RP,-)$ is naturally isomorphic to $Q\hotimes_R-$.
\end{enumerate}
\end{theorem}

\begin{proof}
\eqref{EquiPQ} implies \eqref{EquiF} follows from the associativity of $\hotimes$ (\Cref{OtimesAssociative}) and the fact that $- \hotimes_R \widehat{R}$ is isomorphic to the identity functor (\Cref{RGorroEquiv}).

\eqref{EquiF} implies \eqref{EquiPQ}, \eqref{EquiQ} and \eqref{EquiP}.
Suppose that $F: \rRGX \rightarrow \rSHY$ and
$T: \rSHY \rightarrow \rRGX$ define an equivalence. In particular, both are left and right adjoint and hence, by Proposition \ref{proptensadj}, there exist an $(X,Y)$-bigraded $(R,S)$-bimodule $P$ and a $(Y,X)$-bigraded $(S,R)$-bimodule $Q$ such that $F \cong - \hotimes_{R} P$ and $T \cong - \hotimes_S Q$.
Moreover, we may take $P=\oplus_{x\in X} F((x)R)$ and $Q=\oplus_{y\in Y}T((y)S)$, by \Cref{LeftAdjointP}.

Since $\Id_{\rRGX}$ and $T\circ F$ are naturally isomorphic, using \Cref{TensorShift}, we get for every $x\in X$ isomorphisms in $\rRGX$ as follows:
$$(x)R \simeq (x)R \hotimes_R P \hotimes_{S} Q \simeq {_xP} \hotimes_{S} Q.$$
They yield an isomorphism 
$$\widehat{R} = \oplus_{x \in X} (x)R \simeq \oplus_{x \in X} (_xP \hotimes_{S} Q) = P \hotimes_{S} Q$$
of $X$-graded $R$-bimodules, and analogously, $\widehat{S} \simeq Q \hotimes_R P$ as $Y$-graded $S$-bimodules.
This proves that \eqref{EquiPQ} holds.

By the uniqueness of adjoints, \Cref{proptensadj} implies that $F$ is naturally isomorphic to $H(Q_R,-)$ and $T$ is naturally isomorphic to $H(P_S,-)$.

By \Cref{GCat}, $\oplus_{y\in Y}(y)S$ is a generator of $\rSHY$ and $v(y)S$ is finitely generated and projective in $\rSHY$ for every $v\in V$. Moreover, $Q_R=T(\oplus_{y\in Y}(y)S)$ and $v(_yQ)\cong T(v(y)S)$. Hence $Q_R$ is a generator of $\rRGX$ and  $v(_yQ)$ is finitely generated projective in $\rRGX$. 
Similarly, $P_S$ is a generator of $\rSHY$ and $u({_xP})$ is finitely generated projective for every $x\in X$ and $u\in U$.

Let $y,y'\in Y$. Then the following maps are isomorphisms of additive groups
    $$\Hom_{\rSHY}((y')S,(y)S) \stackrel{-\hotimes_S Q}{\to}
    \Hom_{\rRGX}((y')S\hotimes_R Q,(y)S\hotimes_R Q) \stackrel{\phi}{\to}
    \Hom_{\rRGX}({_{y'}Q},_yQ).$$
The second isomorphism comes from the isomorphisms $\phi_{y,Q}:(y)S\to {_yQ}$ and $\phi_{y',Q}:(y')S\to {_{y'}Q}$ from \Cref{TensorShift}.
The composition restricts to an isomorphism 
	$${_{y}E(\widehat{S}_S)_{y'}}\to {_{y}E(Q_R)_{y'}}$$
which composed with $\lambda_{\widehat{S}}$ yields an isomorphism
	$$\Phi:{_{y}\widehat{S}_{y'}}\to {_{y}E(Q_R)_{y'}},$$
by \Cref{lambdalemma}.
Observe that if $s\in {_{y}\widehat{S}_{y'}}$ and $q\in {_{y}Q}$, then there is $v\in V$ with $vq=q$. Then
\begin{eqnarray*}
\Phi(s)(q)&=&[\phi((-\hotimes_S Q)(\lambda_{\widehat{S}}(s))](q) =
\phi_{y',Q}((\lambda_{\widehat{S}}(s)\otimes_S Q)(\phi_{y,Q}^{-1}(q))) \\
&=& 
\phi_{y',Q}((\lambda_{\widehat{S}}(s)\otimes_S Q)(v^{y}\otimes q)) =
\phi_{y',Q}(sv^{y}\otimes q) = sv^{y}q=svq=sq=\lambda_Q(s)(q).
\end{eqnarray*}
So $\Phi(s)$ is the restriction of $\lambda_Q(s)$ to $_{y}Q_{y'}$. 
This proves that $\lambda_Q$ is an isomorphism of $Y$-graded $S$-bimodules.
The same proof shows that $\lambda_P$ is an isomorphism.
Hence both \eqref{EquiQ} and \eqref{EquiP} hold.

\eqref{EquiQ} implies \eqref{EquiF}
Let $Q$ satisfy the conditions of \eqref{EquiQ} and let $T = - \hotimes_{S} Q : \rSHY \rightarrow \rRGX$.
For $N\in\rSHY$ define the $S$-linear map 
    $$\alpha_N:N\to (F\circ T)(N)=H(Q_R,N\hotimes_S Q)$$
which on $q\in Q$ and $n\in N_y$ with $y\in Y$ acts as follows:
    $$\alpha_N(n)(q)=n\otimes {_yq}.$$
It is easy to see that $N\mapsto \alpha_N$ defines a natural transformation $\Id_{\rSHY}\to F\circ T$. 
For every $y\in Y$ let $\phi_{y,Q}:(y)S\hotimes_S Q \to {_yQ}$ be the isomorphism from \Cref{TensorShift}.
Then $\alpha_{(y)S}$ is the composition of the following maps:
\[\xymatrix{\ar[rr]^-{\lambda_Q|_{(y)S}} (y)S &&  {_yE(Q_R)}=H(Q_R,{_yQ})
	\ar[rr]^-{H(Q_R,\phi_{y,Q}^{-1})} && H(Q_R,(y)S\hotimes_S Q)=(F\circ T)((y)S)}\]
Observe that we have used \Cref{fingen}\eqref{fingenE} in the equality at the middle term. 
Thus $\alpha_{(y)S}$ is an isomorphism for every $y$. 
Since $T$ commutes with direct sums, and by \Cref{fingen}\eqref{fingenCommutes}, so does $F$, if $L$ is a direct sum of modules of the form $(y)S$, then $\alpha_L$ is an isomorphism. 
By \Cref{GCat}, there is an exact sequence 
    $$L_2\to L_1\to N\to 0$$
with $L_1$ and $L_2$ direct sums of modules of the form $(y)S$. 
Moreover, $T$ is right exact and, the assumption implies that $F$ is exact, by \Cref{HExact}. Hence we have the following commutative diagram with exact rows 
    \[\xymatrix{L_2 \ar[d]^{\alpha_{L_2}} \ar[r] & L_1 \ar[d]^{\alpha_{L_1}} \ar[r] & N \ar[d]^{\alpha_{N}} \ar[r] & 0 & \\ H(Q_R, L_2 \hotimes_{S} Q) \ar[r] & H(Q_R, L_1 \hotimes_{S} Q) \ar[r] & H(Q_R, N \hotimes_{S} Q) \ar[r] & 0}\]
Since $\alpha_{L_1}$ and $\alpha_{L_2}$ are isomorphisms so is $\alpha_N$.

Furthermore, consider for each $M \in \rRGX$ the homomorphism of $Y$-graded right $S$-modules
	\[\beta_M : (T\circ F)(M)=H(Q_R,M) \hotimes_S Q \rightarrow M, \quad f \otimes q \mapsto f(q).\]
Observe that the following diagram is commutative, where $\mu:\widehat{S}\hotimes_S Q\to Q$ is the isomorphism from \Cref{TensorShift} and we have used again \Cref{fingen}\eqref{fingenE}
$$\xymatrix{\widehat{S}\hotimes_S Q \ar[rr]^-{\lambda_Q\otimes \Id_Q} \ar[rd]_{\mu}&& E(Q_R)\hotimes_S Q=H(Q_R,Q)\hotimes_S Q \ar[ld]^{\beta_Q} \\
&Q}$$
Then $\beta_Q$ is an isomorphism. Now the same argument as above, using that $Q_R$ is a generator of $\rRGX$ and $F$ is exact and commutes with direct sums, shows that $\beta_M$ is an isomorphism for every $M\in \rRGX$
This shows that $\alpha:\Id_{\rSHY}\to F\circ T$ and $\beta:T\circ F\to \Id_{\rRGX}$ are natural isomorphisms and hence $F$ and $T$ define an equivalence of categories.

\eqref{EquiP} implies \eqref{EquiF}. Suppose that $P$ satisfies \eqref{EquiP} and let $F'=H(P_S,-)$. Using that \eqref{EquiQ} implies \eqref{EquiF} with $P$ taking the role of $Q$, it follows that $F'$ is an equivalence of categories. Then $-\hotimes_R P$ is an equivalence of categories, as it is the left adjoint of $F'$. Thus $F$ is an equivalence of categories.

Finally, if the equivalent conditions \eqref{EquiF}-\eqref{EquiP} hold, then \eqref{EquiQI}-\eqref{EquiIzq} follow from the symmetry of \eqref{EquiPQ}.
\end{proof}

\section{\underline{Restriction and extension of scalars}}\label{SectionRestExt}

In this section we introduce a graded version of restriction and extension of scalars. It turns out that, as in the ungraded case, the extension functor is the left adjoint of the restriction of scalars, and the latter has a right adjoint. We also study when these functors define an equivalence of categories.

The initial data for restriction and extension of scalars in Ring Theory is a ring homomorphism. In our context we need also a groupoid homomorphism and a map connecting the grading sets, compatible with the gradings. More precisely, an \emph{admissible triple}, in the context of \Cref{convention}, is formed by 
\begin{itemize}
	\item a ring homomorphism $\rho:R\to S$, 
	\item a groupoid homomorphism $\gamma:G\to H$, and
	\item a map $\chi:X\to Y$ 
\end{itemize}
satisfying the following conditions for every $g\in G$, $e\in G_0$ and $x\in X_{t(g)}$:
\begin{equation}\label{AdmissibleCond}
\rho(R_g)\subseteq S_{\gamma(g)}, \quad \chi(X_e)\subseteq Y_{\gamma(e)}, \quad
    \chi(x\cdot g)=\chi(x)\cdot \gamma(g).
\end{equation}

\emph{In the remainder of the section  $(\rho,\gamma,\chi)$ is an admissible triple in the context of \Cref{convention}.}

Before defining restriction of scalars we must advise that the naive definition of restriction yields modules that may not be unital or graded. 
Indeed, if $N \in \rSHY$, then even if $N$ is unital as $S$-module, $N$ might not be unital as $R$-module. 
For example, $S$ may not be unital as $R$-bimodule. Actually,  $S$ is unital as $R$-bimodule if and only $\rho(\E(R_{G_0}))$ is a set of local units of $S$, equivalently if $\rho(U)$ is a set of local units of $S$.
Moreover, in the naive $X$-grading $N_x=N_{\chi(x)}$, the $N_x$'s may intersect non-trivially if $\chi$ is not injective. 
To deal with these issues we define the \emph{restriction of scalars functor}
    $$\Res:\rSHY\to \rRGX$$
as follows: The action of $\Res$ on objects associates $N \in \rSHY$ with
	\[\Res(N)=\oplus_{x\in X} \Res(N)_x,\]  
where 
	\[\Res(N)_x=N_{\chi(x)}\cap N\rho(R_{E_x})=\{n\in N_{\chi(x)}:n=n\rho(u^x) \text{ for some } u\in U\}.\]
The right side equality is a consequence of the fact that $U$ is a set of local units of $R$. Indeed, as $u^x\in R_{E_x}$, the right side term is contained in the left side one. For the reversal inclusion take  $n\in N_{\chi(x)}\cap N\rho(R_{E_x})$. Then $n=\sum_{i=1}^k n_i\rho(r_i)$ with each $n_i\in N_{y_i}$ and $r_i\in R_{e_i}$ for some $y_i\in Y$ and $e_i\in E_x$ such that $y_i\cdot \gamma(e_i)=\gamma(x)$. As $U$ is a set of local of $R$, there is $u\in U$ such that $r_i=r_iu$ for each $i$ and in fact $r_i=r_iu_{e_i}=r_iu^x$. Then $n\rho(u^x)=\sum_{i=1}^k n_i\rho(r_iu^x)=\sum_{i=1}^k n_i\rho(r_i)=n$.

Observe that if $x,x'\in X$, and $n\in N_{\chi(x)}\cap N\rho(R_{E_{x}}) \cap N_{\chi(x')}\cap N\rho(R_{E_{x'}})$, then $n$ has a copy  $n_x\in\Res(N)_x$ and another one $n_{x'}\in \Res(N)_{x'}$. Then $n_x=n_{x'}$ if and only if $x=x'$ or $n=0$.
The multiplication in $\Res(N)$ by elements of $R$ is given as follows: If $x\in X$, $g\in G$, $n\in N_{\chi(x)}\cap N\rho(R_{E_{x}})$ and $r\in R_g$, then 
	\[n_x r = \begin{cases} (n\rho(r))_{x\cdot g},& \text{if } x\cdot g \text{ is defined}; \\
		0, & \text{otherwise}\end{cases}\]
This makes sense because if $x\cdot g$ is defined and $u\in U$ with $r=ru$, then $d(g)\in E_{x\cdot g}$ and $r=ru_{d(g)}=ru^{x\cdot g}$. Thus  $n\rho(r)=n\rho(r)\rho(u^{x})\in N_{\chi(x\cdot g)}\cap N\rho(R_{E^{x\cdot g}})$.

In order to verify that $\Res(N)$ is unital as $R$-module let us fix a finite subset $F$ of $\Res(N)$. We may assume without loss of generality that every element of $F$ is homogeneous. So $F=\{(n_1)_{x_1},\dots,(n_m)_{x_m}\}$ with each $n_i\in N_{\chi(x_i)}\cap N\rho(R_{E_{x_i}})$ for each $i$. We can further assume without loss of generality that each $n_i=m_i\rho(r_i)$ with $m_i\in N_{y_i}$, $y_i\in Y$, $r_i\in R_{e_i}$, $e_i\in E_{x_i}$ and $y_i\cdot \gamma(e_i)=\chi(x_i)$. 
	Now take an idempotent $u\in U$ such that $r_iu=r_i$ for every $i$. Then, for every $i\in I$, we have  $r_iu=r_iu_{e_i}$ and 
	\[(n_i)_{x_i}u = \sum_{e\in E_{x_i}} (n_i)_{x_i}u_e = \sum_{e\in E_{x_i}} (m_i\rho(r_i)\rho(u_e))_{x_i}=(m_i\rho(r_iu_{e_i}))_{x_i}=(m_i\rho(r_i))_{x_i}=(n_i)_{x_i}.\]
	
The action of $\Res$ on morphisms associates $\nu\in \Hom_{\rSHY}(N,N')$ with the additive morphism $\Res(\nu):\Res(N)\to \Res(N')$ mapping  $n_x$ to $\nu(n)_x$, 
whenever $x\in X$ and $n\in \Res(N)_{\chi(x)}\cap N\rho(R_{E_x})$.

\begin{proposition}
Consider the $(Y,X)$-bigraded $(S,R)$-bimodule 
$$Q=\bigoplus_{x\in X, y\in Y} {_yQ_x}, \text{ with} \quad  
{_yQ_x}=\{s\in \oplus_{h\in H, y\cdot h = \chi(x)} S_h: s=s\rho(u^x) \text{ for some }u\in U\}$$
Then $\Res$ is exact, naturally isomorphic to $-\hotimes_S Q$ and left adjoint of  $H(Q_R,-)$. 
\end{proposition}

\begin{proof}	
It is easy to check that $\Res$ preserves direct sums and $Q=\oplus_{y\in Y} \Res((y)S)$. Thus, by \Cref{proptensadj} and \Cref{LeftAdjointP}, to prove the proposition it suffices to show that $\Res$ is an exact functor. 
To that end take an exact sequence $N'\stackrel{f}{\to} N \stackrel{g}{\to} N''$ in $\rSHY$ and let $(n)_x\in \Res(N)_x\cap \ker (\Res(g))$ with $n\in N_{\chi(x)}$ and $n=n\rho(u^x)$ for some $u\in U$. 
Then $n\in \ker g$ and so $n=f(n')$ for some $n'\in N'_{\chi(x)}$. Then 	
	\[n'\rho(u^x)=\sum_{e\in E_x} n'\rho(u_e) \in \sum_{e\in E_x} N'_{\chi(x)}S_{\gamma(e)}\subseteq \sum_{e\in E_x} N'_{\chi(x)\cdot \gamma(e)} \subseteq \sum_{e\in E_x}  N'_{\chi(x\cdot e)}=N'_{\chi(x)}\] 
	and $\Res(f)((n'\rho(u^x))_x) = (f(n'\rho(u^x)))_x=(f(n'))_x u^x = (n)_x u^x = (n\rho(u^x))_x =(n)_x$. This shows that $\Res(N')\stackrel{f}{\to} \Res(N) \stackrel{g}{\to} \Res(N'')$ is an exact sequence, and so we have proved that $\Res$ is an exact functor.
\end{proof}

The \emph{extension of scalars}
    $$\Ind :\rRGX \to \rSHY$$
is defined in the standard fashion:
    $$\Ind(M)=M\otimes_R S.$$
for $M\in \rRGX$.
The first issue with restriction of scalars does not show up here because we are assuming that $S$ has graded local units and therefore $\Ind(M)$ is unital as $S$-module.
The $Y$-grading in $\Ind(M)$ is defined as follows:
    \[\Ind(M)_y = \left\{\sum_{i=1}^k m_i\otimes s_i : m_i\in M_{x_i}, s_i\in S_{h_i} \text{ with } h_i\in H,  x_i\in \chi^{-1}(Y_{t(h_i)}) \text{ and }\chi(x_i)\cdot h_i=y\right\}.\]
We should verify that $\Ind(M)=\oplus_{y\in Y} \Ind(M)_y$. That the sum is direct follows by standard arguments but that the equality holds is not completely obvious.
The key observation is that if $m\in M_x$, $s\in S_h$ and $m\otimes s\ne 0$, then $\chi(x)\in Y_{t(h)}$ and hence $m\otimes s\in \Ind(M)_{\chi(x)\cdot h}$. Indeed, take $u\in U$ with $m=mu$. Then $m=mu^x$ and $u^x=\sum_{e\in E_x} u_e$.
Thus $0\ne m\otimes s=m\otimes \sum_{e\in E_x} \rho(u_e)s$ and so $\rho(u_e)s\ne 0$ for some $e\in E_x$ and $\rho(u_e)\in S_{\gamma(e)}$. Therefore $\gamma(e)=t(h)$ and $x\in X_e$, and hence $\chi(x)\in Y_{\gamma(e)}=Y_{t(h)}$, as desired.

\begin{proposition}\label{ExtensionRestrictionAdjoint}
$\Ind$ is left adjoint to $\Res$.
\end{proposition}

\begin{proof}
Let $M\in \rRGX$ and $N\in \rSHY$. Then we define maps
$$\Phi:\Hom_{\rSHY}(\Ind(M),N)\to \Hom_{\rRGX}(M,\Res(N))$$
and
$$\Psi:\Hom_{\rRGX}(M,\Res(N))\to \Hom_{\rSHY}(\Ind(M),N)$$
setting
$$\Phi(\phi)(m)=(\phi(m\otimes \rho(u)))_x \qand
\Psi(\psi)(m\otimes s)=\iota_x(\psi(m))s,$$
for $\phi\in \Hom_{\rSHY}(\Ind(M),N)$, $m\in M_x$, with $x\in X$, $\psi \in \Hom_{\rRGX}(M,\Res(N))$, $s\in S_h$ with $h\in H$ satisfying $\chi(x)\in Y_{t(h)}$ and $u\in U$ satisfying $m=mu$. 
(Observe that $\phi(m\otimes \rho(u))=\phi(m\rho(u^x)\otimes \rho(u))=\phi(m\otimes \rho(u))\rho(u^x)\in N_{\chi(x)}\cap N\rho(R_{E_x})$.) 
Here $\iota_x$ is the map $\Res(N)_x\to N$ associating $n_x$ to $n$ for every $n\in N_{\chi(x)}\cap N\rho(R_{E_x})$.
It is straightforward to check that the definition of $\Phi$ is independent of the election of $u$ and that $\Phi$ and $\Psi$ are natural on both variables and inverse to each other.
\end{proof}

For the remainder of this section, we shall analyze the restriction and extension of scalars functors $\Res$ and $\Ind$ for specific choices of admissible triples.

\begin{example}[Classical restriction and extension of scalars]\label{exclassicalrestriction}
Let $G$ be a subgroupoid of a groupoid $H$ and let $S$ be an $H$-graded ring.
Then $R=S_G = \bigoplus\limits_{g \in G} S_g$ is a subring of $S$ and $R_g=S_g$ defines a $G$-grading on $R$ so that $U=\{v_{G_0}: v\in V\}$ is a set of graded local units of $S_G$.
\begin{enumerate}
\item Consider the regular right actions of $G$ and $H$ on themselves. Then
$\rRGX = \grr R$ and $\rSHY = \grr S$ (with $X=G$ and $Y=H$) and the inclusion maps $\iota_{S_G}:R\to S$ and $\iota_G:G\to H$ yield an admissible triple $(\iota_{S_G},\iota_G,\iota_G)$.
The restriction functor for this triple, which we denote $\Res^S_G$, is the standard grading restriction, 
	\[\Res^S_G : \grr S \to \grr R, \quad N\mapsto \Res^S_G(N)=N_G=\sum_{g\in G} N_g\]
The induction functor, which we denote $\Ind^S_G$, is the standard extension of scalars:
\[\Ind^S_G: \grr R \to \grr S, \quad M\to M\otimes_{R_G} S\]
with grading $(M\otimes_{R_G} S)_h= \bigoplus\limits_{g \in G, t(h) = t(g)} M_g \otimes S_{g^{-1} h}$.

Observe that if $h\in H$ is such that $t(h)\not\in G_0$, then $m\otimes s=0$ for every $m\in M$ and $s\in S_h$. Indeed, as $M_R$ is unitary by assumption, there is $u\in U$ such that $m=mu$. Then $m\otimes s=m\otimes us=0$, as $us=0$.

\item A slight modification of the previous example arises taking the triple $(\iota_{S_G},\iota_G,\chi_{H_G})$ where $H_G=\{h \in H: d(h) \in G_0\}$, with $G$ acting by right multiplication on $H_G$, and $\iota_{H_G}: H_G \rightarrow H$ again the inclusion map. 
In this case we denote $\res^S_G=\Res$ and $\ind^S_G=\Ind$, so $\res^S_G(N)=N_{H_G}$ and $\ind^S_G(M)=M\otimes_{R_{H_G}} S$.
Still, $\res^S_G(N)_x = N_x$ for $N\in \grr S$ and $x\in H_G$, but now $\ind^S_G(M)_h = \bigoplus\limits_{x \in H_G, t(x) = t(h)} M_{x} \otimes S_{x^{-1} h}$ for $M\in \rRGX$ and $h\in H$.

If $H_0 = G_0$, then $H_G = H$, and $\res^S_G$ and $\ind^S_G$ are the standard restriction and extension of scalars.
\end{enumerate}

An important instance occurs when $G = H_0$. In this case $\Res^S_{H_0}(N)=N_{H_0}$, considered as $H_0$-graded $R_{H_0}$-module, $H_G=H$ and $\res^S_{H_0}(N)=N$, considered as $H$-graded $R_{H_0}$-module. 
\end{example}

\begin{example}[Forgetful functors]\label{ForgetfulExtension}{\rm
Let $G$ be a groupoid and let $X$ be a right $G$-set. Let $\sim$ be an equivalence relation on $X$ compatible with the action in the following sense:
	\[\text{for every } g\in G \text{ and } x,x'\in X_{t(g)}, \text{ if } x\sim x' \text{ then } x\cdot g\sim x'\cdot g.\]
Let $\chi:X\to Y=X/\sim$ be the natural map associating $x\in X$ with the class containing $x$. Then $Y$ is a $G$-set with
    $$Y_e=\{\chi(x) : x\in X_e\}$$
for $e\in G_0$, and
    $$y\cdot g = \chi(x\cdot g)$$
for every $g\in G$, $y\in Y_{t(g)}$ and $x\in y\cap X_{t(g)}$. Then $(\Id_R,\Id_G,\chi)$ is an admissible triple. In this case, using the isomorphism $M\otimes_R R\cong M$ it follows that $\Ind$ is naturally isomorphic to a kind of partial forgetful functor associating $M\in \rRGX$ with $M\in \grr (R,G,X/\sim)$ with the $Y$-grading
    $$M_y=\bigoplus_{x\in y} M_x.$$
The functor $\Res:\grr (R,G,X/\sim)\to \rRGX$ associates $N$ with $\Res(N)=\oplus_{x\in X} N_{\chi(x)}$ with $X$-grading $\Res(N)_x=N_{\chi(x)}$ and product defined as follows: If $g\in G$, $r\in R_g$, $x\in X$ and $n\in N_\chi(x)$, then $[n]_{x}$ denotes the element of $\Res(N)_x$ equal to $n$ and
$$[n]_xr =\begin{cases} [nr]_{x\cdot g}, & \text{if } x\in X_{t(g)}; \\
0, & \text{otherwise}.\end{cases}$$

Some cases of this example are of particular interest:

\begin{enumerate}
\item Let $K$ be a subgroupoid of $G$ and consider the following set
    $$X=G_K=\{x\in G : t(x)\in K_0\}$$
Then $XG\subseteq X$ and hence we can consider $X$ as a right $G$-set. Furthermore the following defines an equivalence relation in $X$:
    $$x_1\sim x_2, \text{ if } x_1x_2^{-1}\in K.$$
The equivalence class containing $x$ is $\chi(x)=Kx$. Observe that if $G_0=K_0$, i.e. $K$ is a wide subgroupoid of $G$, then $X=G$  and $X/\sim=K\backslash G$ (see \Cref{ExGroupoids}\eqref{subgroupoid}).

 \item Consider the following equivalence relation on $G$: $g\sim h$ if $d(g)=d(h)$. Then $G/\sim\cong G_0$ and the natural map $\chi:G\to G/\sim$  can identify with the domain map $d:G\to G_0$. Then $\Ind(M)_e \cong  \bigoplus\limits_{g\in G: d(g) = e} M_g$, and $\Res(N)_g= N_{d(g)}$, for each $g \in G$.

\item Let $G$ and $H$ be groupoids, let $\chi: X=G \rightarrow Y=H$ be a surjective groupoid homomorphism, and consider $G$ acting on $X$ by right multiplication and on $Y$ by setting $y \cdot g = y\chi(g)$ for $g\in G$ and $y\in Y$. 
Then, $(\Id_R,\Id_G,\chi)$ is an admissible triple and the induction and restriction are the functors
$$\Ind:\grr R \to \grr (R,G,H) \qand \Res:\grr (R,G,H) \to \grr R,$$
with
$$\Ind(M)_h=\bigoplus\limits_{x,g \in G: \chi(xg) = h} M_x \otimes R_g \cong \bigoplus\limits_{x\in X: \chi(x) = h} M_x,$$
for $M\in \grr R$ and $h\in H$, and 
$$\Res(N)_g= N_{\chi(g)},$$    
for $N\in \grr (R,G,H)$ and $g\in G$. 

If $|H|=1$, then $\Ind:\grr R\to \Mod\text{-}R$ is the classical forgetful functor $\mathcal{F}$ and hence $\Res$ is isomorphic to the right adjoint $\mathcal{G}$ of $\mathcal{F}$ (see \Cref{EjemploForgethful}).
\end{enumerate}
}
\end{example}

\section{\underline{Applications}}

In this section we obtain two applications of \Cref{Equivalencia}. The first one characterizes when the functors of restriction and extension of scalars are equivalences of categories. The second one focuses on deciding when a category of graded modules is equivalent to a category of ungraded modules. 

\subsection{When restriction and extension of scalars are equivalences of categories?}

In this subsection $(\rho,\gamma,\chi)$ is an admissible triple, in the context of \Cref{convention}, and $\Res$ and $\Ind$ are the corresponding restriction and extension of scalars functors.
We start obtaining a description of $\Ind((x)R)$ as a $Y$-graded submodule of $(\chi(x))S$. 

\begin{lemma}\label{ExtenShift}
The map
    $$\mu:R\otimes_R S\to \rho(R)S, \quad r\otimes s \mapsto \rho(r)s$$
is an isomorphism of right $S$-modules.
For every $x\in X$, $\mu$ restricts to an isomorphism
    $$\mu_x:\Ind((x)R) \to \rho((x)R)S=\{s \in (\chi(x))S : s=\rho(u^x)s \text{ for some } u\in U\}$$
in $\rSHY$, where the target of the latter is considered as $Y$-graded submodule of $(\chi(x))S$.
\end{lemma}

\begin{proof}
Obviously $\mu$ is a homomorphism of right $S$-modules.
That $\mu$ is bijective follows from standard arguments using that $R$ has local units.

We firstly prove that $\mu_x$ is a morphism in $\rSHY$.
For that it suffices to show that if
$r\in R_g\setminus \{0\} \subseteq (x)R_{x'}$ and $s\in S_h$ with $\chi(x')\cdot h=y$, then $\rho(r)s\in (\chi(x))S_y$.
Indeed, since $(x)R_{x'}\ne 0 $, we have $x\cdot g=x'$ and hence $\chi(x)\cdot \gamma(g)=\chi(x')$. If $\gamma(g)h$ is not defined, then $\rho(r) s = 0 \in (\chi(x))S_y$, as desired. Otherwise, i.e., if $\gamma(g) h$ is defined, then $\rho(r)s\in S_{\gamma(g)h}$ and  
	\[ \chi(x)\cdot(\gamma(g)h) = (\chi(x)\cdot \gamma(g))\cdot h =\chi(x')\cdot h=y.\] 
Therefore, $\rho(r)s\in (\chi(x))S_y$, as desired. 

As clearly $\mu(\Ind((x)R))=\rho((x)R)S$, it suffices to prove that the latter is equal to the following set: $A=\{s \in (\chi(x))S : s=\rho(u^x)s \text{ for some } u\in U\}$.
Since $u^x\in (x)R_x$ for every $u\in U$, obviously, $A \subseteq \rho((x)R)S$.
To prove the reverse inclusion take $r\in (x)R$.
If $g$ belongs to the support of $r$, then $x\in X_{t(g)}$. Thus $\chi(x)\in Y_{\gamma(t(g))}=Y_{t(\gamma(g))}$.
As $\rho(r_g)\in S_{\gamma(g)}$, it follows that $\rho(r_g)\in (\chi(x))S$.
Thus $\rho(r)=\sum_{g\in G}\rho(r_g) \in (\chi(x))S$. Therefore $\rho((x)R)S\subseteq (\chi(x))S$.
Let $u\in U$ satisfying $ur=r$.
Then $u^x\in (x)R$ and $u^xr=r$.
Thus $\rho(r)s=\rho(u^x)\rho(r)s$ for every $s\in S$.
This shows that $\rho((x)R)S\subseteq A$.
\end{proof}

By \Cref{ExtenShift} we can identify $\Ind((x)R)_y$ and $\rho((x)R)S_y$, which we denote  ${_xP_y}$. We use these additive groups to define an $(X,Y)$-bigraded $(R,S)$-bimodule: 
\begin{equation}\label{PResExt}
P=\bigoplus_{x\in X, y\in Y} \rho((x)R)S_y.
\end{equation}
Observe that ${_xP}=\rho((x)R)S$, which we are identifying with $\Ind((x)R)$ as $Y$-graded $S$-module.
By \Cref{LeftAdjointP} and \Cref{ExtensionRestrictionAdjoint}, $P$ is an $(X,Y)$-graded $(R,S)$-bimodule and $\Ind$ is naturally isomorphic to $-\hotimes_R P$.
The next two lemmas characterize when $\{{_xP} :x\in X\}$ is a set of generators of $\rSHY$, and when $\lambda_P$ is bijective. Both are  necessary conditions for $\Ind$ to be an equivalence of categories, by \Cref{Equivalencia}.

\begin{lemma}\label{Tastgenerator}
Let $P$ be the bimodule defined in \eqref{PResExt}. Then, the following conditions are equivalent:
	\begin{enumerate}
		\item\label{PGen} $\{{_xP}:x\in X\}$ is a set of generators of $\rSHY$.
		\item\label{AllUnits} Every $v\in \E(S_{H_0})$ satisfy the following condition:
		\begin{itemize}
			\item[$(\bullet)$] For every $y\in Y$ there are $u\in U$, $h_1,\dots,h_n\in H$ and $x_1,\dots,x_n\in X$, such that $y=\chi(x_i)\cdot h_i$ for every $i=1,\dots,n$ and $v^y\in \sum_{i=1}^n S_{h_i^{-1}}\rho(u^{x_i}) S_{h_i}$.		
		\end{itemize}
		\item\label{OneSetLU} $S$ has a set of graded local units $S$ satisfying condition $(\bullet)$.
		\end{enumerate}
\end{lemma}

\begin{proof}
\eqref{AllUnits} implies \eqref{OneSetLU} is clear.

\eqref{PGen} implies \eqref{AllUnits}. Let $v\in \E(S_{H_0})$ and $y \in Y$.
Then $v^y\in (y)S_y$.
By assumption there exist $x_1,\dots x_n \in X$ and a morphism $f: \bigoplus\limits_{i = 1}^n \rho((x_i)R)S \rightarrow (y)S$
in $\rSHY$ such that $v^y \in \Imagen f$. Write $f = \sum_{i=1}^n f_i$, with $f_i\in \Hom_{\rSHY}(\rho((x_i)R)S,(y)S)$. Then there exist $s_i \in \rho((x_i)R)S$, for $i = 1,\dots, n$, such that $v^y = \sum_i f_i(s_i)$. Moreover, we may assume that each $s_i$ is homogeneous, say of degree $h_i$, as element of the $H$-graded ring $S$.
As $v^y\in (y)S_y$, we also may assume that $s_i$ has degree $y$ as element of $(\chi(x_i))S$, i.e. $\chi(x_i)\cdot h_i=y$.
By the description of $\rho((x_i)R)S$ in \Cref{ExtenShift}, for every $i=1,\dots,n$ there is $u_i\in U$ such that $\rho(u_i^{x_i})s_i=s_i$.
Let $u\in U$ such that $u_iu=u_i$ for every $i=1,\dots,n$.
Then $u_i^{x_i}u^{x_i}=u_i^{x_i}$ 
for every $i=1,\dots,n$ and
	\[v^y=\sum_{i=1}^n f_i(\rho(u_i^{x_i})\rho(u^{x_i})s_i) = 
	\sum_{i=1}^n f_i(\rho(u_i^{x_i}))\rho(u^{x_i})s_i.\]
Furthermore, the support of $v$ is contained in $H_0$ and $\rho(u^{x_i})s_i\in S_{h_i}$ for every $i$. Therefore, if $s'_i$ is the component of degree $h_i^{-1}$ of  $f_i(\rho(u_i^{x_i}))$, then
$v^y=\sum_{i=1}^n s'_i\rho(u^{x_i})s_i\in \sum_{i=1}^n S_{h_i^{-1}}\rho(u^{x_i})S_{h_i}$, as desired.

\eqref{OneSetLU} implies \eqref{PGen}. Let $V$ be a set of graded local units of $S$ satisfying $(\bullet)$.
By \Cref{GCat}, $\{(y)S:y\in Y\}$ is a set of generators of $\rSHY$. Therefore, it suffices to prove that for each $y\in Y$, every $a\in (y)S$ belongs to the image of some $f\in \Hom_{\rSHY}(\oplus_{i=1}^n {_{x_i}P},(y)S)$ with $x_1,\dots,x_n\in X$.
Let $v\in V$ such that $a=va$. By assumption there are $u\in U$, $h_i\in H$, $x_i\in X$, $s_i\in S_{h_i}$ and $s'_i\in S_{h_i^{-1}}$ for $i=1,\dots, n$ such that $\chi(x_i)\cdot h_i=y$ for every $i=1,\dots,n$ and $v^y=\sum_{i=1}^n s'_i \rho(u^{x_i})s_i$.
Then $v^y=\sum_{i=1}^n \lambda_{s'_i}(\rho(u^{x_i})s_i)$, where $\lambda_{s'_i}\in \Hom_{\rSHY}((\chi(x_i))S,(y)S)$ is as defined in \eqref{leftmultmap}. 
Let $f_i$ be the restriction of $\lambda_{s'_i}$ to $_{x_i}P=\rho((x_i)R)S$ and let $f:\prod_{i=1}^n {_{x_i}P}\to (y)S$ be the morphism in $\rSHY$ given by the coproduct property. Let $b=(\rho(u^{x_1})s_1,\dots,\rho(u^{x_n})s_n)\in \prod_{i=1}^n {_{x_i}P}$. Then $f(b)=v^y$ and hence $f(ba)=v^ya=a$.
\end{proof}

\begin{lemma}\label{rhobijective}
    Let $P$ be the bimodule defined in \eqref{PResExt}. Then,
    \begin{enumerate}
        \item\label{rhoInj} $\lambda_P$ is injective if and only if the restriction of $\rho$ to $(x)R_{x'}$ is injective for each $x, x' \in X$.
        \item\label{rhoSur} $\lambda_P$ is surjective if and only if $\rho(u^x)\left( (\chi(x))S_{\chi(x')}\right)\rho(u^{x'})\subseteq \rho((x)R_{x'})$ for each $u\in U$ and $x, x' \in X$.
    \end{enumerate}
\end{lemma}
\begin{proof}
\eqref{rhoInj} Suppose that $\lambda_P$ is injective. Let $x, x' \in X$ and $r \in (x)R_{x'} \cap \Ker \rho$. Then, for $p \in P$
    $$\lambda_P(r)(p) = \rho(r)\cdot {_{x'}p} = 0,$$
    and therefore, $\lambda_P(r) = 0$, which implies $r = 0$, by assumption. Thus, the restriction of $\rho$ to $(x)R_{x'}$ is injective.

    Conversely, suppose now that $\rho|_{(x)R_{x'}}$ is injective for every $x, x' \in X$. Let $r \in {_x(\Ker \lambda_P)_{x'}}$ and take $u \in U$ such that $ru = r$. Then, $ru^{x'}=r$, $\rho(u^{x'})\in {_{x'}P}$ and
    $$0 = \lambda_P(r)(\rho(u^{x'})) = \rho(r)\rho(u^{x'}) = \rho(ru^{x'}) = \rho(r).$$
    By assumption $r = 0$. This proves that $\lambda_P$ is injective.

\eqref{rhoSur} Suppose that $\lambda_P$ is surjective and let $a = \rho(u^x)s\rho(u^{x'})$ with $u\in U$, $x,x'\in X$ and $s \in (\chi(x))S_{\chi(x')}$.
If $p\in \;_{x'}P$, then $p=\rho(^{x'}u_1)s'$ for some $u_1\in U$ and $s'\in S$. Then $ap=\rho(u^x)s\rho(u^{x'}u_1^{x'})s'\in \; _xP$. Hence,  
    $$\phi_{a}: {_{x'}P} \longrightarrow {_xP}, \quad p \mapsto ap.$$
defines an element of $H(P_S,{_xP})_{x'}$ which actually belongs to $_xE(P_S)_{x'}$,  as $u\phi_{a}u = \phi_a$. By assumption there exists $r \in {_x\widehat{R}_{x'}}=(x)R_{x'}$ such that $\lambda_P(r) = \phi_a$, and therefore
    $$\rho(ru^{x'}) = \rho(r)\rho(u^{x'}) = r\cdot \rho(u^{x'})= \lambda_P(r)(\rho(u^{x'})) = \phi_a(\rho(u^{x'})) = a\rho(u^{x'}) = a.$$
    Moreover, $ru^{x'} \in (x)R_{x'}$, and hence $a \in \rho((x)R_{x'})$, as desired.

Conversely, suppose that $\rho(u^x)\left((\chi(x)S)_{\chi(x')}\right)\rho(u^{x'})\subseteq \rho((x)R_{x'})$ for each $u\in U$ and $x, x' \in X$.
Let $f \in {_xE(P_S)_{x'}}$. Then $ufu = f$ for some $u \in U   $. Moreover, $\Imagen f\subseteq {_xP}\subseteq (\chi(x))S$ and $f(\bigoplus_{z\in X\setminus \{x'\}}{_zP})=0$. Thus $f=u^xfu^{x'}$. Moreover, $\rho(u^{x'})\in (\chi(x'))S_{\chi(x')}$ and hence  $f(\rho(u^{x'}))\in {_xP_{\chi(x')}}\subseteq (\chi(x))S_{\chi(x')}$.
Thus
    $$\rho(u^x)  f(\rho(u^{x'})) = \rho(u^x) f(\rho(u^{x'}))\rho(u^{x'}) \in \rho(u^x)\left((\chi(x)S)_{\chi(x')}\right)\rho(u^{x'}).$$
By assumption, there is $r \in (x)R_{x'}$ such that $\rho(r) = \rho(u^x)  f(\rho(u^{x'}))$.
For every $p \in P$, $_{x'}p\in S$ and, as $f\in {_xE(P)}_{x'}$ 
$$f(p)= f({_{x'}p}) = (u^xfu^{x'})({_{x'}p}) = \rho(u^x) f(\rho(u^{x'})\;{_{x'}p})    = \rho(u^x) f(\rho(u^{x'}))\; {_{x'}p}= \rho(r)\; {_{x'}p}=\lambda_P(r)(p).$$
Thus $\lambda_P(r)=f$.
\end{proof}

\begin{theorem}\label{EquivRest}
Let $(\rho:R\to S,\gamma:G\to H, \chi:X\to Y)$ be an admissible triple and let $U$ and $V$ be sets of graded local units of $R$ and $S$, respectively. Then the following conditions are equivalent:
\begin{enumerate}
    \item $\Ind:\rRGX\to \rSHY$ is an equivalence of categories. 
    \item  $\Res:\rSHY\to \rRGX$ is an equivalence of categories.
    \item \begin{enumerate}
    \item\label{EquivRestGen} For every $v\in V$ and every $y\in Y$ there are $u\in U$, $h_1,\dots,h_n\in H$ and $x_1,\dots,x_n\in X$, such that $y=\chi(x_i)\cdot h_i$ for every $i=1,\dots,n$ and $v^y\in \sum_{i=1}^n S_{h_i^{-1}}\rho(u^{x_i}) S_{h_i}$.
    \item\label{EquivRestIso} For each $x, x' \in X$,  $\rho$ restricts to a bijection $(x)R_{x'} \rightarrow \sum_{u\in U} \rho(u^x)\left((\chi(x))S_{\chi(x')}\right)\rho(u^{x'})$.
\end{enumerate}

\end{enumerate}
\end{theorem}

\begin{proof}
By \Cref{ExtensionRestrictionAdjoint}, $\Res$ is right adjoint to $\Ind$ and so (1) and (2) are equivalent. Moreover, by \Cref{LeftAdjointP} and \Cref{ExtenShift}, $\Ind$ is naturally isomorphic to $-\hotimes_R P$, for $P$ the $(X,Y)$-bigraded $(R,S)$-bimodule $P$ defined in \eqref{PResExt}. 
Furthermore, $u({_xP}) = \rho(u) \cdot (\chi(x))S$, for every $x \in X$ and $u \in U$. Then, by \Cref{GCat}, $u\cdot {_xP}$ is a finitely generated projective $Y$-graded right $S$-module. This is one of the conditions in statement \eqref{EquiP} in \Cref{Equivalencia}. So, by that theorem, conditions (1) is equivalent to the following: $P_S$ is a generator of $\rSHY$ and $\lambda_P$ is an isomorphism. 
Then the equivalence between (1) and (3) follows from \Cref{Tastgenerator} and \Cref{rhobijective}.
\end{proof}

By \Cref{Tastgenerator}, Condition \eqref{EquivRestGen} in \Cref{EquivRest} can be replaced by the following: There is a set of local units of $S$ satisfying condition ($\bullet$) in \Cref{Tastgenerator}. 

We finish this subsection applying \Cref{EquivRest} to particular cases of restriction and extension of scalars described in \Cref{exclassicalrestriction}. For expressions of the type $M_{gh}$ where $g$ and $h$ are elements of a groupoid, say $G$, and $M$ is a $G$-graded ring or module, we are implicitly assuming that if $gh$ is not defined then $M_{gh}=0$. This is also valid for the case of modules graded by $G$-sets.

\begin{corollary}\label{ResGEquiv}
Let $G$ be a subgroupoid of a groupoid $H$, let $S$ be a $H$-graded ring. Then, the following are equivalent:
    \begin{enumerate}
        \item $\Res^S_G:\grr S\to \grr S_G$ and $\Ind^S_G:\grr S_G\to \grr S$ define an equivalence of categories;
        \item For every $h \in H$, $S_{d(h)} = \sum_{g \in G} S_{h^{-1} g} S_{g^{-1}h}$;
        \item For every $h,k \in H$, $S_{hk} = \sum_{g \in G} S_{hg} S_{g^{-1}k}$
    \end{enumerate}
In that case $G_0$ contains $d(h)$ and $t(h)$ for every $h\in H$ with $S_h\ne 0$.
\end{corollary}
\begin{proof}
As $d(h)=h^{-1}h$, clearly, (3) implies (2). Furthermore, if (2) hold, then for every $h,k\in H$, $S_{hk}=S_{hk}S_{d(k)} = S_{hk} \sum_{g\in G} S_{k^{-1}g}S_{g^{-1}k} \subseteq \sum_{g\in G} S_{hg} S_{g^{-1}k}$. This proves (2) implies (3).

To prove that (1) and (2) are equivalent, notice that in this case, condition \eqref{EquivRestIso} of \Cref{EquivRest} is always satisfied, as $\iota_{S_G}$ restricts to the identity map of $(x)S_{x'}$, for every $x, x' \in G$. 
So it suffices to prove that condition (2) is equivalent to condition \eqref{EquivRestGen} in \Cref{EquivRest} and for that we can take $U = V \cap S_G$, as set of graded local units of $S_G$.
Suppose that condition \eqref{EquivRestGen} of \Cref{EquivRest} holds. Let $h\in H$ and $s\in S_{d(h)}$. Then $s=vsv$ for some $v\in V$ and therefore $s=v^hsv^h$. 
In this case $h\in H=Y$, $X=G$ and $\gamma=\chi=\iota_G:G\hookrightarrow H$ and $\rho=\iota_{S_G}$. So, by assumption, there are $u\in U=V\cap S_G:S_G\hookrightarrow S$, $h_1,\dots,h_n\in H$ and $g_1,\dots,g_n\in G$ such that $h=g_ih_i$ for every $i=1,\dots,n$ and $v^h\in \sum_{i=1}^n S_{h_i^{-1}}u^{g_i}S_{h_i}$. 
Since $S_{h_i^{-1}} u^{g_i} \subseteq S_{h^{-1}g_i}S_{d(g_i)}\subseteq S_{h^{-1}g_i}$, we have $v^h\in  \sum_{i=1}^n S_{h^{-1}g_i} S_{g_i^{-1}h}\subseteq \sum_{g\in G} S_{h^{-1}g}S_{g^{-1}h}$. 
Moreover, $s\in S_{d(h)}$ and therefore 
\[s=v^hsv^h\in \sum_{g,g_1\in G} S_{h^{-1}g}S_{g^{-1}h}S_{d(h)}S_{h^{-1}g_1}S_{g_1^{-1}h}\subseteq \sum_{g\in G} S_{h^{-1}g} S_{g^{-1}h}\]
So the direct inclusion in (2) holds, but the reverse inclusion is obvious. 
Conversely, suppose that (2) holds and let $y\in Y(=H)$ and $v\in V$. Then $v^h\in S_{d(h)}$ and hence, by assumption $v^h=\sum_{i=1}^n a_ib_i$ with $a_i\in S_{y^{-1}x_i}$ and $b_i\in S_{x_i^{-1}y}$ for some $x_1,\dots,x_n\in G(=X)$. Of course we are implicitly assuming that all the $x_i^{-1}y$ are defined, so that $h_i=x_i^{-1}y\in H$. Moreover, there is $u\in U(=S_G\cap V)$ such that $a_iu=a_i$ and hence $a_iu^{x_i}=a_i$, since $d(y^{-1}x_i)=d(x_i)$. Therefore $v^y=\sum_{i=1}^n a_i u^{x_i} b_i \in \sum_{i=1}^n S_{h_i^{-1}}\rho(u^{x_i})S_{h_i}$. This proves that condition \eqref{EquivRestGen} of \Cref{EquivRest} holds.

Suppose that (3) holds and let $h\in H$ with $S_h\ne 0$. Then $\sum_{g\in G} S_{hg}S_{g^{-1}d(h)} = S_{hd(h)}=S_h\ne 0$ and hence $S_{hg}\ne 0$ for some $g\in G$. In particular $d(h)=t(g)\in G_0$. Similarly, $\sum_{g\in G} S_{t(h)g} S_{g^{-1}h}=S_h\ne 0$ and therefore $t(h)=d(g)\in G_0$ for some $g\in G$.
\end{proof}

Specializing \Cref{ResGEquiv} to the case where $G=H_0$ yields the following corollary which shows that $\Res^R_{G_0}$ is an equivalence of categories if and only if $R$ is strongly graded. This generalizes Theorem 3.1.1 in \cite{NastasescuVanOystaeyen}. 

\begin{corollary}\label{strgrrestriction}
The following conditions are equivalent for a $G$-graded ring $R$, with $G$ a groupoid:
    \begin{enumerate}
    \item $\Res^R_{G_0}:\grr R\to \grr R_{G_0}$ and $\Ind^R_{G_0}:\grr R_{G_0}\to \grr S$ define an equivalence of categories.
    \item  $R$ is strongly graded.
    \item $R_{d(g)}=R_{g^{-1}}R_g$ for every $g\in G$.
\end{enumerate}
\end{corollary}

\begin{corollary}\label{ResHGEquiv}
	Let $G$ be a subgroupoid of a groupoid $H$, let $S$ be a $H$-graded ring. Then, the following are equivalent:
	\begin{enumerate}
		\item $\res^S_G$ and $\ind^S_G$ define an equivalence of categories between $gr\text{-}(S_G,G,H_G)$ and $\grr S$;
		\item \begin{enumerate}
			\item\label{ResHGEquivGen} $S_{d(h)}=\sum_{k\in H_G, t(k)=d(h)} S_kS_{k^{-1}}$ for every $h\in H$. 
			\item\label{ResHGEquivIso} If $h\in H_G$ satisfies $t(h)\in G_0$ and $S_h\ne 0$, then $h\in G$.
		\end{enumerate}
	\end{enumerate}
\end{corollary}

\begin{proof} 
By \Cref{EquivRest} it suffices to prove that for the admissible triple $(\iota_{S_G},\iota_G,\iota_{H_G})$, conditions \eqref{EquivRestGen} and \eqref{EquivRestIso} of \Cref{EquivRest} are equivalent to conditions \eqref{ResHGEquivGen} and \eqref{ResHGEquivIso}, respectively. That this is the case for the (a) conditions follows by similar arguments as in the proof of \Cref{ResGEquiv}.

Now suppose that condition \eqref{EquivRestIso} of \Cref{EquivRest} holds for the admissible triple $(\iota_{S_G},\iota_G,\iota_{H_G})$. 
Let $h\in H_G$ such that $t(h)\in G_0$ and $S_h\ne 0$. Let $a\in S_h\setminus 0$. Then there is $v\in V$ such that $a=vav$ and hence $a=v_{G_0}av_{G_0}\in \rho(U)((t(h))S_{h})\rho(U)$. Applying \eqref{EquivRestIso}, and having in mind that in this case $\rho=\iota_{S_G}$, it follows that $a\in (t(h))R_h$. Since $a\ne 0$, there is $g\in G$ such that $h=t(h)g$. Then $h=g\in G$. Thus condition \eqref{ResHGEquivIso} holds.

Conversely, suppose \eqref{ResHGEquivIso} holds and let $U=V\cap S_{G}$, a set of graded local units of $S_G$. 
Let $x,x'\in H_G (=X)$ and let $a\in U((x)R_{x'})U$ with $a\ne 0$.
Thus, there is $v\in V$, $h\in H$ and $b\in R_h$ such that $a=v_{G_0}bv_{G_0}$ and $xh=x'$.
In particular, $d(h)=d(x')\in G_0$, so that $h\in H_G$ and $t(h)=d(x)\in G_0$. Furthermore, 
$0\ne a=v_{t(h)}bv_{d(h)}\in R_h$.  Therefore, $h\in G$, by hypothesis. Thus $a\in (x)R_{x'}$. 
This shows that $(x)R_{x'}=U((x)R_{x'})U$.  
As $\iota_{S_G}$ is an inclusion map, condition \eqref{EquivRestIso} of \Cref{EquivRest} holds for the admissible triple $(\iota_{S_G},\iota_G,\iota_{H_G})$.
\end{proof}

\subsection{When $\rRGX$ is equivalent to a category of modules?}

A positive answer to this question follows at once from \Cref{GCat} and Corollary 1.4 of \cite{AbramsMeninidelRio1994}. 
For completeness we give a short independent proof.

\begin{theorem}\label{EquivRGXmodS}
 Let $R$ be a $G$-graded ring and let $X$ be a $G$-set. Then, $\rRGX$ is equivalent to $\Modr S_X(R)$. In particular, $\rRGX$ is equivalent to the category of (unital) modules of a ring of local units.
\end{theorem}

\begin{proof}
Let $T$ the trivial group, let $S=S_X(R)$ and let $Q=\widehat{R}$. By the paragraph before \Cref{SLocalUnits}, $Q$ has a structure of $(T,X)$-bigraded $(S,R)$-bimodule such that $\lambda_Q$ is an isomorphism.  
By \Cref{GCat}, $Q_R$ is a projective generator of $\rRGX$, and, by \Cref{SLocalUnits}, $S$ has a set of local units $V$ such that $vQ$ is finitely generated for every $v\in V$. Thus $Q$ satisfies the conditions of \Cref{Equivalencia}\eqref{EquiQ}, and hence $H(Q_R,-):\rRGX \to \Modr S$ is an equivalence of categories.
\end{proof}

The following theorem is a generalization of a result of C. Menini and C. Năstăsescu which characterize unital group graded rings whose category of graded modules is isomorphic to the category of modules of a unital ring. 

\begin{theorem}\label{MeniniNastasescu}
	Let $G$ be a groupoid, let $R$ be a $G$-graded ring with a set of graded local units and let $X$ be a $G$-set. Then, the following are equivalent
	\begin{enumerate}
		\item\label{MNEquiv} $\rRGX$ is equivalent to the category of modules over a unital ring.
		\item\label{MNStronglyGraded}  $\rRGX$ is equivalent to $\grr S$, for a unital ring $S$ which is strongly $H$-graded for some groupoid $H$.
		\item\label{MNFGProjGen} $\rRGX$ has a finitely generated projective generator.
		\item\label{MNFGGen} $\rRGX$ has a finitely generated generator.
		\item\label{MNugRGen} There exist an idempotent $u$ of $R_{G_0}$ and a finite subset $F$ of $X$ such that $\{u(x)R:x\in F\}$ is a set of generators of $\rRGX$.
		\item\label{MNRdg} There exist an idempotent $u$ of $R_{G_0}$ and a finite subset $F$ of $X$ such that, for every $x \in X$, $R_{E_x} = \sum_{g\in G, x\cdot g\in F} R_g u R_{g^{-1}}$.
	\end{enumerate}
	Moreover, if $u$ and $F$ satisfy condition \eqref{MNugRGen} and $M=\oplus_{x\in F} u(x)R$, then $\rRGX$ is equivalent to the category of modules over $\End_{\rRGX}(M)$.
\end{theorem}

\begin{proof}
	\eqref{MNEquiv} if and only if \eqref{MNStronglyGraded}. The direct implication follows by simply considering the unital ring graded by the trivial group. The reverse implication is a consequence of 
	\Cref{strgrrestriction} and the fact that if $S$ is a unital $G$-graded ring then $S_{G_0}$ is unital too and $\grr S_{G_0}=\Modr S_{G_0}$.  
	
	\eqref{MNEquiv} if and only if \eqref{MNFGProjGen}. The direct implication is clear and the reverse implication follows from \cite[Theorem~IV.4.1]{Mitchell1965}. 
	
	\eqref{MNFGProjGen} implies \eqref{MNFGGen} is clear. 
	
	\eqref{MNFGGen} implies \eqref{MNugRGen}. Suppose that $M$ is a finitely generated generator of $\rRGX$. 
	By \Cref{RGorro}, $\{u(x)R : x\in X, u\in U_x\}$ is a set of generators of $\rRGX$ and hence there is a surjective graded homomorphism $\oplus_{i=1}^n u_i(x_i)R \to M$ with $x_1,\dots,x_n\in X$ and $u_i\in U_{x_i}$. 
	Then $\{u_i(x_i)R : i=1,\dots,n\}$ generates $\rRGX$. On the other hand, as $R$ has graded local units, there is an idempotent $u$ of $R_{G_0}$ such that $uu_i=u_i=u_iu$ for every $i=1,\dots,n$. Moreover $u_i(x_i)R$ is a direct summand of $u(x)R$ in $\rRGX$. Therefore $\{u(x_i)R:i=1,\dots,n\}$ is a set of generators of $\rRGX$.
	
	\eqref{MNugRGen} implies \eqref{MNFGProjGen} and last statement. Let $u$ and $F$ satisfy the conditions in \eqref{MNugRGen} and let $M=\oplus_{x\in F} u(x)R$. Then $M$ is a generator of $\rRGX$ which is finitely generated and projective by \Cref{GCat}. Moreover, $\rRGX$ is equivalent to the category of modules over $\End_{\rRGX}(M)$, by \cite[Theorem~IV.4.1]{Mitchell1965}.
	
	\eqref{MNugRGen} if and only if \eqref{MNRdg}. Fix an idempotent $u$ of $R$ and a finite subset $F$ of $X$. We will show that $u$ and $F$ satisfy \eqref{MNugRGen} if and only if they satisfy \eqref{MNRdg}.  
	
	Suppose that $u$ and $F$ satisfy \eqref{MNugRGen} and let $x \in X$. If $g\in G$ and $x\cdot g$ is defined then $R_guR_{g^{-1}} = R_g u_{d(g)} R_{g^{-1}} \subseteq R_{t(g)}\subseteq R_{E_x}$. 
	Thus $\sum_{g\in G, x\cdot g\in F} R_g u R_{g^{-1}}\subseteq R_{E_x}$. 
	To prove the reverse inclusion, let $e\in E_x$ and $r\in R_e$. Then $r\in (x)R_x$ and by the assumption there are $z_1,\dots,z_n\in F$ such that $r$ belongs to the image of a graded homomorphism $f:\oplus_{i=1}^q u(z_i)R\to (x)R$.
	So $r=\sum_{i=1}^q f(ur_i)$ with $r_i\in (z_i)R_x$ and we may assume without loss of generality that each $r_i$ is homogeneous, say of degree $h_i$, as element of $R$, with $z_i\cdot h_i=x$. Then $ur_i=u_{t(h_i)}ur_i$ and $u_{t(h_i)} \in (z_i)R_{z_i}$. Therefore $f(u_{t(h_i)})\in (x)R_{z_i}$ and so $f(u_{t(h_i)})=\sum_{j=1}^{m_i} r'_{i,j}$ with each $r'_{i,j}\in R_{k_{j,i}}$ for $k_{j,i}\in G$ such that $x\cdot k_{j,i}=z_i$. 
	Then $r=\sum_{i=1}^q f(u_{t(h_i)})ur_i = \sum_{i=1}^q \sum_{j=1}^{m_i} r'_{i,j}ur_i$. But $r$ is homogeneous of degree $e$ and if $r'_{i,j}ur_i\ne 0$, then it is homogeneous of degree $k_{j,i}h_i$. So dropping the summands on which $k_{j,i}h_i$ is not defined or different of $e$ we obtain $r=\sum_{i=1}^n r'_iur_i$ with $r_i\in R_{h_i}$, $r'_i\in R_{g_i}$ and $e=g_ih_i$. Thus $h_i=g_i^{-1}$ and $x\cdot g_i=z_i\in F$. Therefore, $r\in \sum_{g\in G, x\cdot g\in F} R_guR_{g^{-1}}$. 
	This shows that $u$ and $F$ satisfy \eqref{MNRdg}.
	
	Conversely, suppose that $u$ and $F$ satisfy \eqref{MNRdg}. 
	Let $v\in U$ and $x\in X$. 
	Then $v^x\in R_{E_x}$. 
	By hypothesis, $v^x=\sum_{i=1}^n r'_i u r_i$ with $r_i \in R_{g_i^{-1}}$, $r'_i \in R_{g_i}$ for some $g_i\in G$ with $z_i=x\cdot g_i\in F$. 
	Then each $r_i'ur_i\in R_{d(g_i)}$ 
	Since $v^x$ is idempotent, we can suppose that $r'_i = v^x r'_i$, for each $i = 1, \dots, n$. Now, define $f_i: u(z_i)R \rightarrow v(x)R$ such that $f_i(r) = r'_i r$. 
	It is easy to see that $f_i$ is well defined and is a homomorphism of $G$-graded right $R$-modules. Furthermore, 
	$f = \bigoplus_{i = 1}^n f_i: \bigoplus_{i = 1}^n u(z_i)R \rightarrow v(x)R$ is a surjective homomorphism of $G$-graded right $R$-modules, as $f((ur_i)_i) = v^x$ and $v(x)R=v^xR$. Since $\{v(g)R: g \in G, v \in U(R)\}$ is a set of generators of $\rRGX$, so is $\{u(g)R:g\in F\}$, i.e. $u$ and $F$ satisfy condition \eqref{MNugRGen}.
\end{proof}

\textbf{Acknowledgments}. 
The authors are grateful to Javier Sánchez for his valuable comments and to the referee of the first version for the careful report pointing out several embarrassing errors. 
The first author thanks the Departamento de Matemáticas of Universidad de Murcia for the hospitality during his stay along the year 2024 when this paper was mostly written.

\textbf{AI disclosure}
The first version of this paper was written without any use of AI. 
After correcting the errors pointed out by the referee we asked AI (OpenAI, GPT 5.6) for a revision of the new version. The output pointed out some minor imprecise statements and several typos and suggested to add some additional explanations in some of the proofs. We fixed the imprecise statements and typos and add some of the suggested explanations, but not all as some of them were elementary folklore. None of the mathematical ideas in the paper come from the interaction with any AI.

\bibliographystyle{amsplain}
\bibliography{References}

\end{document}